\PassOptionsToPackage{dvipsnames}{xcolor}

\documentclass[onefignum,onetabnum]{siamart220329}

\usepackage{iftex}
\ifPDFTeX 
  \usepackage[utf8]{inputenc} 
  \usepackage[T1]{fontenc} 
  \usepackage{lmodern} 
\else
  \usepackage{fontspec} 
\fi

\usepackage[english]{babel} 
\usepackage{hyperref}

\usepackage{cite}                           
\usepackage{subfigure}

\usepackage{tikz}
\usepackage{pgfplots}
\pgfplotsset{compat=newest,compat/show suggested version=false}
\usetikzlibrary{pgfplots.groupplots}

\renewcommand{\phi}{p}
\def\R{\mathbb R}
\def\P{\mathcal P}
\def\proj{\mathcal P}
\def\f{\hat f}
\def\h{\hat h}
\def\LL{\mathcal L}
\DeclareMathOperator{\spann}{span}

\def\be{\begin{equation}}
\def\ee{\end{equation}}
\def\bea{\begin{eqnarray}}
\def\eea{\end{eqnarray}}
\def\beas{\begin{eqnarray*}}
\def\eeas{\end{eqnarray*}}

\usepackage{lipsum}
\usepackage{amsfonts}
\usepackage{graphicx}
\usepackage{epstopdf}
\usepackage{algorithmic}
\ifpdf
  \DeclareGraphicsExtensions{.eps,.pdf,.png,.jpg}
\else
  \DeclareGraphicsExtensions{.eps}
\fi

\newsiamremark{remark}{Remark}
\newsiamremark{hypothesis}{Hypothesis}
\crefname{hypothesis}{Hypothesis}{Hypotheses}
\newsiamthm{claim}{Claim}

\usepackage{amsopn}

\title{Conservative polynomial approximations and applications to Fokker-Planck equations}
\author{Tino Laidin\thanks{Univ. Lille, CNRS, Inria, UMR 8524 - Laboratoire Paul Painlevé, F-59000 Lille, France (\email{tino.laidin@univ-lille.fr}).}\and Lorenzo Pareschi\thanks{Maxwell Institute for Mathematical Sciences and Department of Mathematics, Heriot-Watt University, Edinburgh, United Kingdom (\email{l.pareschi@hw.ac.uk}) and Department of Mathematics and Computer Science, University of Ferrara, Italy (\email{lorenzo.pareschi@unife.it}).}}

\begin{document}

\maketitle

\begin{abstract}
 We address the problem of constructing approximations based on orthogonal polynomials that preserve an arbitrary set of moments of a given function without loosing the spectral convergence property. To this aim, we compute the constrained polynomial of best approximation for a generic basis of orthogonal polynomials. The construction is entirely general and allows us to derive structure preserving numerical methods for partial differential equations that require the conservation of some moments of the solution, typically representing relevant physical quantities of the problem. These properties are essential to capture with high accuracy the long-time behavior of the solution. We illustrate with the aid of several numerical applications to Fokker-Planck equations the generality and the performances of the present approach.   

   \vspace{.5em}
   \textsc{Keywords:} Fokker-Planck equation, Galerkin spectral method, conservative methods, spectral accuracy.\\\vspace{.5em}
   \textsc{2020 Mathematics Subject Classification: }65N35 (Primary), 
   82C40 (Secondary). 
\end{abstract}

\tableofcontents

\section{Introduction}
Computational techniques that maintain certain fundamental properties or structures of the underlying mathematical problem in their discrete approximations play a crucial role in the study and analysis of ODEs and PDEs, offering insights into complex physical phenomena that are often inaccessible through analytical means alone^^>\cite{Gosse2013,HWL2006,Jin2022}. These structure-preserving properties could include symmetries, conservation laws, or other structural characteristics that are crucial for accurately representing the behavior of the system being modeled. 

Capturing the long-time behavior of solutions to PDEs is also intricately linked to the structure-preserving properties of numerical methods. In fact, preserving key structural properties at a discrete level, such as conservation laws and invariant quantities, ensures that the numerical approximation retains the essential characteristics of the original problem over extended time intervals. 

In this regard, the preservation of moments of the solution stands out as a fundamental aspect, especially when considering kinetic and mean-field equations, where the moments represent the macroscopic observable physical quantities^^>\cite{DimarcoPareschi2015,PareschiZanella2018,BessemoulinFilbet2021,Kraus2017}. In kinetic equations, for example, the long time behavior of the systems is fully determined by the knowledge of some invariant moments. Achieving such properties, however, is particularly challenging due to the inherent complexity of many physical systems and the necessity to accurately capture these dynamics.

A substantial body of literature has been devoted to the development of structure-preserving methods for various types of PDEs, with a focus on preserving specific physical properties. For instance, in the context of Fokker-Planck equations, numerous studies have explored different numerical strategies aimed at maintaining conservation of mass, momentum and energy with the aim to describe accurately the steady state solution of the problem^^>\cite{PareschiZanella2018, Bailo2020, Buet2010, Larsen1985, BorziMohammadi2015, Chang1970}.

Spectral methods, which rely on expansions in terms of orthogonal polynomials, have garnered significant attention for their effectiveness in solving collisional kinetic equations of Boltzmann-type. The spectral accuracy and efficiency of these methods make them particularly well-suited for capturing the non-equilibrium dynamics of such systems^^>\cite{Pareschi2000, Pareschi2000a, Pareschi2000b, Hu2019, MouhotPareschi2006, Mouhot2013}. 

However, due to the lack of conservations, the long time behavior of such methods may lead to accumulation of errors and  their extension to a conservative setting has represented a challenging task in numerical simulations. Recent advancements in conservative spectral methods based on $L^2$-minimization frameworks have shown promise in addressing these challenges^^>\cite{PareschiRey2022, Gamba2010, Gamba2017,Alonso2018}. 

Building upon these developments, our main goal in this manuscript is to extend the $L^2$-minimization framework to derive families of orthogonal polynomials capable of preserving the moments of the the solution, thereby enabling the construction of spectrally accurate and efficient structure-preserving numerical methods. This extension leverages the well-established theory of orthogonal polynomials^^>\cite{canuto1988, funaro1992, Gautschi2004, KoekoekLeskySwarttouw2010, GottliebOrszag1977}, offering a robust framework for achieving high-accuracy approximations of PDEs while preserving key physical properties.
In particular, we have developed a general framework for constructing constrained orthogonal polynomial approximations which maintain the accuracy properties of the unconstrained polynomials for smooth solutions. This framework enabled us to construct spectrally accurate moment-preserving Galerkin-type approximations for several Fokker-Planck equations originating from various fields, ranging from classical physics to social sciences and operations research. Consequently, these equations are defined in different domains, both bounded and unbounded, and possess different steady states, requiring the adoption of suitable orthogonal polynomial bases.

Let us also remark, that the connection between orthogonal polynomials and probability theory to design uncertainty quantification methods further underscores the importance of our approach, as classical families of continuous and discrete orthogonal polynomials are intimately linked to probability distributions and our considerations naturally generalize to such contexts^^>\cite{Xiu2010}.

In the following sections, we provide a detailed exposition of our methodology and present numerical results demonstrating the efficacy of our approach. The rest of the manuscript is organized as follows.
In Section 2, we introduce some notations and present our moment-preserving approach based on a constrained $L_2$-minimization setting. We also study the convergence properties and prove a general result on spectral accuracy.
Section 3 is then devoted to testing the novel polynomial approximation, first by approximating probability densities with given moments and subsequently by considering different Fokker-Planck equations on bounded and unbounded domains.
Some concluding remarks are provided in the last section.

\section{Conservative approximations by orthogonal polynomials}
\label{sec:specContraint}
In order to get conservative approximations, we construct a conservative projection on the space of orthogonal polynomials by generalizing the constrained formulation approach introduced in^^>\cite{PareschiRey2022} for trigonometric polynomials.
In particular, we will give an explicit formulation of the orthogonal polynomial of best approximation in the weighted least square sense, constrained by preservation of moments, and show that it preserves a spectral accuracy property for smooth solutions like the classical orthogonal polynomial approximation. 

\subsection{Orthogonal polynomials}
	Let us first set up the mathematical framework of our analysis. To simplify our treatment we will consider the one-dimensional case. Given a real function $f(x) \in L^2_\omega(\Omega)$, with $\Omega \subseteq \R$ and $\omega(x) > 0$ a positive function, we define
\be
\|f\|_{L^2_\omega} = \left(\int_{\Omega}	|f(x)|^2 \omega(x)\,dx\right)^{1/2},
\label{eq:pnorm}
\ee	
which has the associated inner product
\be
\langle f,g \rangle_\omega = \int_{\Omega}	f(x)g(x) \omega(x)\,dx.
\ee
We assume that $\omega(x) > 0$ is such that
\be
\int_{\Omega} |x|^q\,\omega(x)\,dx < \infty,\quad q=0,1,2,\ldots
\label{eq:omega}
\ee
We consider the problem of approximating functions in the weighted $L^2_\omega(\Omega)$ space defined by the norm \eqref{eq:pnorm}. In the case $\omega(x) = 1$ we will use standard notations $\|\cdot\|_{L^2}$ and $\langle \cdot,\cdot \rangle$ to denote the $L^2$ norm and the associated inner product.

A polynomial basis $\{\phi_k\}$, $k=0,1,\ldots$ for $L^2_\omega(\Omega)$ is a set of functions such that any $f$ in the
space can be expressed uniquely as
\be
f(x) = \sum_{k=0}^\infty \f_k \phi_k(x).
\label{eq:frapp}
\ee
A set of polynomials $\{\phi_k\}$ is called \emph{orthogonal} if $\langle \phi_h,\phi_k \rangle_\omega = 0$ for $h\neq k$. We refer to Table \ref{tab:poly} for some examples of standard polynomial families

Orthogonal bases are particularly nice, both for theory and numerical approximation. In fact, from
\[
\langle f,\phi_k \rangle_\omega = \sum_{h=0}^\infty \f_h \langle \phi_h,\phi_k \rangle_\omega = \f_k \langle \phi_k,\phi_k \rangle_\omega,
\]
the coefficients in \eqref{eq:frapp} are easily represented as 
\be
\f_k = \frac{\langle f,\phi_k \rangle_\omega}{\langle \phi_k,\phi_k \rangle_\omega} = \frac{\langle f,\phi_k \rangle_\omega}{\|\phi_k\|^2_{L^2_\omega}}.
\ee
A first result that we recall is that the truncated approximation 
\be
f_N(x) = \sum_{k=0}^N \f_k \phi_k(x)
\ee	
is the best approximation of $f$ in the subspace $S_N={\spann}\left\{\phi_k\,|\, k=0,1,\ldots,N\right\}$.
More precisely, let $\proj_{\omega,N}: L^2_{\omega}(\Omega) \rightarrow S_N$ be the
orthogonal projection upon $S_N$ in the inner product of
$L^2_{\omega}(\Omega)$ 
\be
\langle f-\proj_{\omega,N} f,\phi_k\rangle_{\omega}=0,\qquad \forall\,\, \phi_k\,\in\,S_N.
\label{eq:ortho}
\ee
With these definitions $\proj_{\omega,N} f=f_N$, is the solution of the following minimization problem
\[
f_N = {\rm argmin} \left\{\| g_N - f \|_{L^2_{\omega}}\,\,:\,\,g_N\in S_N\,\right\}.
\]
By orthogonality, one also obtains the Parseval's identities 
\begin{equation*}
\| f \|^2_{L^2_\omega} = \sum_{k=0}^\infty |\f_k|^2\| \phi_k \|^2_{L^2_\omega}
,\qquad \| f_N \|^2_{L^2_\omega} = \sum_{k=0}^N |\f_k|^2\| \phi_k \|^2_{L^2_\omega},
\end{equation*}
which gives the error estimate
\[
\| f - f_N \|^2_{L^2_\omega} = \sum_{k=N+1}^\infty |\f_k|^2\| \phi_k \|^2_{L^2_\omega}.
\]	

An important feature of the orthogonal polynomial approximations on $S_N$ is related to their spectral convergence properties for smooth solutions. We report here a result for Jacobi type approximations^^>\cite{funaro1992}
\begin{theorem}
\label{thm:ClassicalFourierSpectralAccuracy}
If $f\in H_\omega^r(\Omega)$, where $r \geq 0$ is an integer, then there exists a constant $C > 0$ dependent on $\alpha,\beta$ and $r$ such that
 \be
\|f-f_N\|_{{L^2_\omega}} \leq \frac{C}{N^r} \left\|(1-x^2)^{r/2}\frac{d^r f}{dx^r}\right\|_{L^2_\omega},\qquad N > r.
\label{eq:ses}
\ee
\end{theorem}

\begin{table}[htp]
\caption{Families of orthogonal polynomials}
\begin{center}
\begin{tabular}{l|cc}
\hline\hline
Name & $\omega(x)$ & $\Omega$ \\
\hline\hline\\[-.3cm]
Jacobi &$(1-x)^\alpha (1+x)^\beta$, $\alpha,\beta > -1$ &$[-1,1]$\\[+.1cm]
\quad - Legendre & $1$, $\alpha=\beta=0$ &$[-1,1]$\\[+.1cm] 
\quad - Chebyshev 1st kind & $\displaystyle\frac1{\sqrt{1-x^2}}$, $\alpha=\beta=-\displaystyle\frac12$ &$[-1,1]$\\[+.3cm]
\quad - Chebyshev 2nd kind & $\displaystyle{\sqrt{1-x^2}}$, $\alpha=\beta=\displaystyle\frac12$ &$[-1,1]$\\[+.3cm]
Laguerre & $x^\alpha e^{-x}$, $\alpha > -1$& $\R^+$ \\[+.1cm]
Hermite & $|x|^{2\alpha} e^{-x^2}$, $\alpha > -1$ & $\R$ 
\end{tabular}
\end{center}
\label{tab:poly}
\end{table}%

\subsection{Approximation of moments}

Next, we discuss some properties of the projection operator $\proj_{\omega,N}$, in particular those concerning approximation of moments of nonnegative functions. We will denote for $q=0,1,2,\ldots$
\be
m_q = \int_{\Omega} f(x) x^q\,dx = \left\langle f, {x^q}\right\rangle, \qquad
m_{q,N} = \int_{\Omega} f_N(x) x^q\,dx = \left\langle f_N, {x^q}\right\rangle.
\ee
Note that
\[
m_{q} = \sum_{k=0}^\infty \f_k \mu_{q,k}, \qquad m_{q,N} = \sum_{k=0}^N \f_k \mu_{q,k},
\]
where for $k,q \geq 0$
\be
\mu_{q,k}=\langle \phi_k, x^q\rangle,
\label{eq:momp}
\ee
is the $q$-th moment of the $k$-th order polynomial.

It is well-known that moments of the weight function $\omega(x)$ permit to generate explicitly the orthogonal polynomials\cite{KoekoekLeskySwarttouw2010}. In our case, however, we are interested in the behavior of the moments of the orthogonal polynomials themselves defined in \eqref{eq:momp}. An important characterization of monic orthogonal polynomials 
is the classical three-term recurrence relation\cite{funaro1992, Gautschi2004, KoekoekLeskySwarttouw2010, GottliebOrszag1977}
\be
\begin{split}
\phi_{k}(x)&=(a_k x + b_k)\phi_{k-1}(x)+c_{k}\phi_{k-2}(x),\quad k \geq 1, 
\end{split}
\label{eq:three}
\ee
where we assumed $\phi_0(x)=1$, $\phi_{-1}(x)=0$ and the sequences $\{a_k\}_{k\geq 1}$, $\{b_k\}_{k\geq 1}$ and $\{c_k\}_{k\geq 1}$ depend on the particular polynomial family under consideration. 

Equation \eqref{eq:three} permits to recursively compute the value of $\phi_k(x)$ starting from the values of $\phi_0(x)$ and $\phi_1(x)$. 

From \eqref{eq:three} one gets the following relations for the moments coefficients 
\be
\begin{split}
\mu_{q,k} =& a_k\mu_{q+1,k-1}+b_k\mu_{q,k-1}+c_k\mu_{q,k-2},\quad k \geq 1,\quad q\geq 0,
\end{split}
\label{eq:momr}
\ee
with 
\[
\mu_{q,0}=\frac{x^{q+1}}{q+1}\Big|_{\Omega},\quad \mu_{q,-1}=0,\quad q \geq 0,
\]
which, therefore, can also be computed recursively from \eqref{eq:momr} up to the moment $q$ provided we know $q+h$ moments of $\phi_{k-h}(x)$ for $h=1,\ldots, k$. 

\subsubsection{Unbounded domains}
Note, however, that for unbounded domains $\Omega$, like the case of Laguerre or Hermite polynomials (see Table \ref{tab:poly}), the moments $\mu_{q,0}$, $q \geq 0$, are not finite. In such cases, one resorts on suitable orthogonal function families with respect to the inner product $\langle\cdot,\cdot\rangle$ defined in the symmetrically weighted case^^>\cite{ManziniFunaroDelzanno2016, BessemoulinFilbet2021,Holloway1996} as 
\[
P_k(x) = \phi_k(x) w^{1/2}(x),
\] 
so that $\langle P_h, P_k \rangle = \langle p_h, p_k \rangle_\omega$, $\forall\, h,k \geq 0$.
 
For functions that we write in the form $f(x)=h(x)w^{1/2}(x)$ we have
\be
f(x)=\sum_{k=0}^\infty \h_k P_k(x),
\label{eq:fscaled}
\ee
with
\be
\h_k = \frac{\langle f,P_k \rangle}{\langle P_k,P_k \rangle} = \frac{\langle h,\phi_k \rangle_\omega}{\|\phi_k\|^2_{L^2_\omega}}.
\ee
Clearly, the orthogonal functions $P_k$ satisfy an analogous three term relation as \eqref{eq:three} with $P_0(x)=w^{1/2}(x)$, and similarly their moments satisfy \eqref{eq:momr} with 
\[
\langle P_0, x^q\rangle =\frac{x^{q+1}}{q+1}w^{1/2}(x)\Big|_{\Omega},\quad \mu_{q,-1}=0,\quad q \geq 0.
\] 
The above quantity is bounded for Laguerre and Hermite polynomials thanks to \eqref{eq:omega}. 

We can then define $f_N(x)$ as the truncated approximation to the first $N+1$ terms of \eqref{eq:fscaled}. We have
\be
\langle f-f_N,P_k\rangle=0,\qquad \forall\,\, P_k=\phi_k w^{1/2}, \phi_k\,\in\,S_N.
\label{eq:ortho2}
\ee
In the weighted function family representation Parseval identity reads
\[
\| f \|_2 = \sum_{k=0}^\infty |\h_k|^2 \|P_k\|^2_2.
\]
We report here a result for Hermite approximations^^>\cite{funaro1992,FokGuoTang2002}
\begin{theorem}
\label{thm:HermiteSpectralAccuracy}
If $f\in H_\omega^r(\Omega)$, where $r \geq 0$ is an integer, then there exist a constant $C > 0$ and $r$ such that
 \be
\|f-f_N\|_{{L^2_\omega}} \leq \frac{C}{N^{r/2}} \left\|f\right\|_{H^r_\omega},\qquad N > r.
\label{eq:SpecralHermite}
\ee
\end{theorem}

Let us remark that, in general, moments are not preserved by the projection. Neither in standard nor in the weighted case. In fact, we have
\be\label{eq:MomProjError_weight}
m_q - m_{q,N} =  \left\langle f - f_N, {x^q}\right\rangle  = \left\langle f - f_N, \frac{x^q}{\omega(x)}\right\rangle_\omega 
\ee
which, by \eqref{eq:ortho}, is equal to zero only if ${x^q}/{\omega(x)}$ belongs to $S_N$.
A special case is represented by the Legendre polynomials which by construction preserve all moments of the approximated function $f(x)$ up to the order $q \leq N$.

Finally, we have for each $\varphi \in L^2_\omega(\Omega)$ by H\"older's inequality
\begin{equation}\label{eq:WeightedSpectralMoment}
  |\left\langle f,\varphi\right\rangle_\omega - \left\langle f_N,\varphi\right\rangle_\omega| \leq \|\varphi\|_{L^2_\omega} \|f-f_N\|_{L^2_\omega}.
\end{equation}
Therefore, the projection error on the moments decays faster than algebraically, for the weighted norms, when the
solution is infinitely smooth. Although this is in general a guarantee of the accuracy of the methods, in practice the loss of conservations can lead to accumulations of error in the approximation of PDEs over long time scales, resulting in the determination of inaccurate steady states\cite{PareschiZanella2018,Gosse2013}.   

\subsection{Moment preserving orthogonal polynomials}
Although the error on moments is spectrally small for smooth functions, in many applications moments represent structural properties of the PDE and are related to conservation of physical properties such as mass, momentum and energy. In such cases, the design of a numerical method that is structure preserving requires the construction of a projection operator that does not affect the relevant moments. To this end, we will construct a different projection operator on the space of polynomials, $\proj^c_{\omega,N}: L^2_\omega(\Omega)\to S_N$ such that it preserves a finite number of moments. More precisely, we require the projection to satisfy for $q=0,1,\ldots,Q$
\[
\left\langle \proj^c_{\omega,N} f, {x^q}\right\rangle = \left\langle f, {x^q}\right\rangle,
\]
but maintaining the convergence properties of the original orthogonal projection.

To this aim, it is natural to consider the following constrained best approximation problem 
\be
\label{def:constrainedApprox}
f^c_N = {\rm argmin} \left\{\| g_N - f \|^2_{L^2_\omega}\,:\,g_N\in S_N,\,\, \langle g_N,x^q \rangle  =
\langle f,x^q \rangle,\,\,q=0,1,\ldots,Q\right\}.
\ee 
Now, since $g_N\in S_N$ we can represent it in the form
\[
g_N(x)=\sum_{k=0}^N \hat g_k \phi_k(x)
\]
and then by Parseval's identity
\[
\| g_N - f \|^2_{L^2_\omega} = \sum_{k=0}^\infty |\hat g_k-\hat f_k|^2\|\phi_k\|^2_{L^2_\omega},
\]
where we assumed $\hat g_k=0$, $k > N$.

We require that
\be
\langle g_N,x^q \rangle = \sum_{k=0}^N \hat g_k \mu_{q,k} =\langle f,x^q \rangle,\quad q=0,1,\ldots,Q
\label{eq:mom}
\ee
or, after introducing the vector of moments $\Phi = (1,x,x^2,\ldots,x^Q)^T$, in vector form 
\[
\langle g_N,\Phi \rangle = \sum_{k=0}^N \hat g_k \hat \Phi_k =\langle f, \Phi \rangle,
\]
where $\hat \Phi_k=(\mu_{0,k},\mu_{1,k},\ldots,\mu_{Q,k})^T$.

Let us now solve the minimization problem \eqref{def:constrainedApprox} using the Lagrange multiplier method.
Let $\lambda\in\R^{Q+1}$ be the vector of Lagrange multipliers, we consider the objective function 
\[
\LL(\hat g,\lambda) = \sum_{k=0}^\infty |\hat g_k-\hat f_k|^2\|\phi_k\|^2_{L^2_\omega} + \lambda^T\left(\sum_{k=0}^N \hat g_k \hat \Phi_k -\langle f, \Phi \rangle\right)=0,
\]
with $\hat g\in\R^{N+1}$ the vector of coefficients $\hat g_k$, $k=0,\ldots,N$.  

Stationary points are found by imposing
\begin{eqnarray*}
\frac{\partial \LL(\hat g, \lambda)}{\partial \hat g_k} = 0,\quad k=0,1,\ldots,N;\qquad
\frac{\partial \LL(\hat g, \lambda)}{\partial \lambda_q} = 0,\quad q=0,1,\ldots,Q.
\end{eqnarray*}
From the first condition, one gets
\begin{equation}
	\label{eq:cond1Lagrange}
	2({\hat g}_k-{\hat f}_k)\|\phi_k\|^2_{L^2_\omega}+ \lambda^T \hat\Phi_k=0,
\end{equation}
whereas the second condition corresponds to \eqref{eq:mom}.

Multiplying the above equation by $\hat \Phi_k/\|\phi_k\|^2_{L^2_\omega}$ and summing up over $k$ we can write
\[
2 \sum_{k=0}^N ({\hat g}_k-{\hat f}_k)\hat \Phi_k +   \sum_{k=0}^N \frac1{\|\phi_k\|^2_{L^2_\omega}} \hat \Phi_k\hat\Phi^T_k\lambda=0
\]
and, using  \eqref{eq:mom} as well as the fact that $\hat \Phi_k\hat\Phi^T_k$
are symmetric and positive definite matrices of size $Q+1$ one obtains
\be
	\label{eq:Lambda}
	\lambda = -2\left(\sum_{k=0}^N \frac1{\|\phi_k\|^2_{L^2_\omega}} \hat \Phi_k\hat\Phi^T_k\right)^{-1}(\langle f, \Phi \rangle-\langle f_N, \Phi \rangle).
\ee 
Now, rewriting the first condition \eqref{eq:cond1Lagrange} as
\[
	{\hat g}_k = {\hat f}_k - \frac{1}{2\|\phi_k\|^2_{L^2_\omega}}  \hat\Phi_k^T \lambda,
\]
and plugging the expression \eqref{eq:Lambda} of $\lambda$ in it, one obtains that the minimum is achieved for $\hat g_k = \hat f_k^c$, given by the following definition.

\begin{definition}
For $f \in L^2_\omega(\Omega)$, we define the \emph{conservative orthogonal projection} $\proj_{N,\omega}^c f = f_N^c$ in $S_N$, where $f_N^c$ is given by
\be
f_N^c(x) = \sum_{k=0}^N \hat f^c_k \phi_k(x),
\label{eq:cp}
\ee
and we define the \emph{moment constrained coefficients} as
\be
{\hat f^c_k} ={\hat f_k} + \hat C^T_k (\langle f, \Phi \rangle-\langle f_N, \Phi \rangle),\qquad \hat C^T_k = \frac{1}{\|\phi_k\|^2_{L^2_\omega}} \hat \Phi^T_k \left(\sum_{h=0}^N \frac1{\|\phi_h\|^2_{L^2_\omega}} \hat \Phi_h\hat\Phi^T_h\right)^{-1}.
\label{eq:minim1}
\ee
\label{def:1}
\end{definition}

The following result states the spectral accuracy of the conservative best approximation in least square sense \eqref{eq:cp}, and generalizes Theorem \ref{thm:ClassicalFourierSpectralAccuracy}. In order to prove this result, one needs the following hypothesis:
\begin{hypothesis}\label{Hyp:SpectralMoment}
The moments of the approximated solution are spectrally accurate in the classical $L^2$-norm. Namely, there exists $C>0$ such that
\be
\|U-U_N\| \leq \frac{C}{N^r}\|f\|_{H_\omega^r} \|\Phi\|_{2,L^2_\omega}
\ee
where $U$ and $U_N$ denote the moment vectors of $f$ and $f_N$ respectively:
\be
U =\left\langle f,\Phi\right\rangle ,\quad U_N=\left\langle f,\Phi\right\rangle.
\ee
\end{hypothesis}
This assumption is different from the spectral accuracy on the moment, in weighted norm, one could obtain from \eqref{eq:WeightedSpectralMoment}. This assumption is a consequence of a (classical) spectral accuracy in $L^2$ of the function $f$. Therefore, it is a little bit weaker than asking further regularity on $f$ in $L^2$ in addition to the one in weighted $L^2$. This property depends heavily on the function under consideration and we will see in numerical simulations that it may only be a technical detail.
\begin{theorem}
\label{thm:SpectralAccuracyfNc}
If $f\in H_\omega^r(\Omega)$, where $r \geq 0$ is an integer, we have
\be
\| f - f^c_N\|_{{L_\omega^2}} \leq \frac{C_\Phi}{N^r} \|f\|_{H_\omega^r} 
\label{eq:sec}
\ee
where the constant $C_\Phi$ depends on the spectral radius of the matrix 
\begin{equation*}
    M=\sum_{h=0}^N\frac{1}{\|\phi_h\|^2_{L^2_\omega}}\hat \Phi_h\hat\Phi^T_h,
\end{equation*}
 and on $\|\Phi\|^2_{2,L^2_\omega}=\sum_{q=0}^{Q}\|\Phi_q\|^2_{L^2_\omega}$ 
where $\Phi_q$, $q=0,1,\ldots,Q$ are the components of the vector $\Phi$.
\end{theorem}
\begin{proof}
We can split
\[
\| f - f^c_N\|^2_{L^2_\omega} \leq \| f - f_N\|^2_{L^2_\omega}+\| f_N - f^c_N\|^2_{L^2_\omega}.
\]
The first term is bounded by the spectral estimate of truncated approximation of the form \eqref{eq:ses}, whereas
for the second term, by Parseval's identity, we have 
\[
\| f^c_N - f_N \|^2_{L^2_p} = \sum_{k=0}^N |\hat f^c_k-\hat f_k|^2\|p_k\|^2_{L^2_\omega}.
\]
Now, using the definition \eqref{eq:minim1}
we get
\begin{eqnarray*}
\sum_{k=0}^N |\hat f^c_k-\hat f_k|^2\|p_k\|^2_{L^2_\omega} &=& \sum_{k=0}^N |\hat C^T_k (U-U_N)|^2\|p_k\|^2_{L^2_\omega} \leq \|U-U_N\|_2^2 \sum_{k=0}^N \|\hat C^T_k\|_2^2\|p_k\|^2_{L^2_\omega}.
\end{eqnarray*}
where $\|\cdot\|_2$ denotes the euclidean norm of the vector. The assumption \ref{Hyp:SpectralMoment} of spectral accuracy, in classical $L^2$ norm, on the moments implies
\[
\|U-U_N\|_2 \leq \frac{C}{N^r} \|f\|_{H_\omega^r} \left(\sum_{q=0}^{Q}\|\Phi_q\|^2_{L^2_\omega}\right)^{1/2} = \frac{C}{N^r}\|f\|_{H_\omega^r}\|\Phi\|_{2,L^2_\omega}
\]
where for $\Phi=(\Phi_0,\ldots,\Phi_{Q})^T$ we defined 
$\|\Phi\|^2_{2,L^2_\omega}=\sum_{q=0}^{Q}\|\Phi_q\|^2_{L^2_\omega}$. Finally, by definition of $\hat C^T_k = \frac{1}{\|\phi_k\|^2_{L^2_\omega}} \hat \Phi^T_k \left(\sum_{h=0}^N \frac1{\|\phi_h\|^2_{L^2_\omega}} \hat \Phi_h\hat\Phi^T_h\right)^{-1}$, one has
\begin{equation*}
  \sum_{k=0}^N \|\hat C^T_k\|_2^2\|\phi_k\|^2_{L^2_\omega} \leq \sum_{k=0}^N \|\hat \Phi_k\|_2^2 \left\|\left(\sum_{h=0}^N \frac{1}{\|\phi_h\|^2_{L^2_\omega}} \hat \Phi_h\hat\Phi^T_h\right)^{-1}\right\|^2_2.
\end{equation*}
Now let $M = \sum_{h=0}^N\hat \Phi_h\hat\Phi^T_h/\|\phi_h\|^2_{L^2_\omega}$. It follows that 
\begin{equation}
    \sum_{k=0}^N \|\hat C^T_k\|_2^2\|\phi_k\|^2_{L^2_\omega} \leq \|\Phi\|^2_{2,L_\omega^2}\left\|M^{-1}\right\|^2_2.
\end{equation}
As a sum of positive definite symmetric matrices, $M$ enjoys the same properties. Consequently,
\begin{equation}
    \sum_{k=0}^N \|\hat C^T_k\|_2^2\|\phi_k\|^2_{L^2_\omega} \leq \|\Phi\|^2_{2,L_\omega^2}\rho^2(M^{-1}),
\end{equation}
where $\rho(M^{-1})$ denotes the spectral radius of the matrix $M^{-1}$. Combining everything, one finally obtains
\begin{equation*}
\| f^c_N - f_N \|_{L^2_\omega} \leq  \frac{C}{N^r} \|f\|_{H_\omega^r}\|\Phi\|_{2,L_\omega^2}\rho(M^{-1})
\end{equation*}
which proves \eqref{eq:sec}.
\end{proof}  

\begin{remark}
The conservative best approximation in least square \eqref{eq:minim1} can be represented in term of the standard projection as
\be
\P_N^c f = \P_N f + \sum_{k=0}^N \hat C^T_k \langle f-\P_N f,\Phi \rangle.
\label{eq:standp}
\ee
\end{remark}

\begin{remark}
  The result of Theorem \ref{thm:SpectralAccuracyfNc} guarantees the spectral accuracy of the constrained approximation. It is however worth mentionning that the constant $C_\Phi$ from the error estimate involves the spectral radius of the matrix $M^{-1}$, which itself depends on the number of modes. Nonetheless, this dependency doesn't seem to significantly affect the convergence rate of the constrained method as shown through the various numerical experiments in Section \ref{sec:NumExp}. Note that the spectral radius of $M^{-1}$, which appears in $C_\Phi$, doesn't grow with $N$. It is more affected by a high number of constraints. We present in Figure \ref{fig:SpectralRad} this spectral radius as a function of $N$ and $Q$ in the case of the Chebyshev 2\textsuperscript{nd} kind polynomials.
  \begin{figure}[!ht]
    \centering
    \includegraphics[width=.45\linewidth]{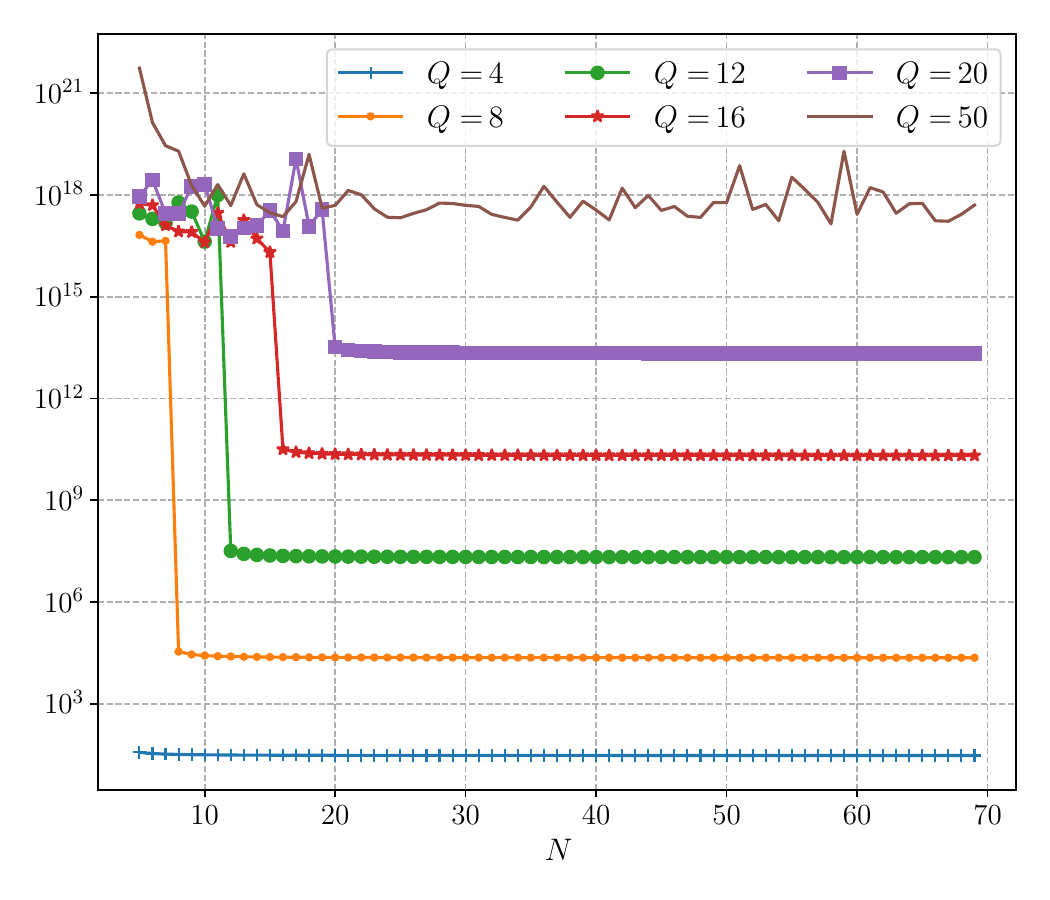}
    \includegraphics[width=.45\linewidth]{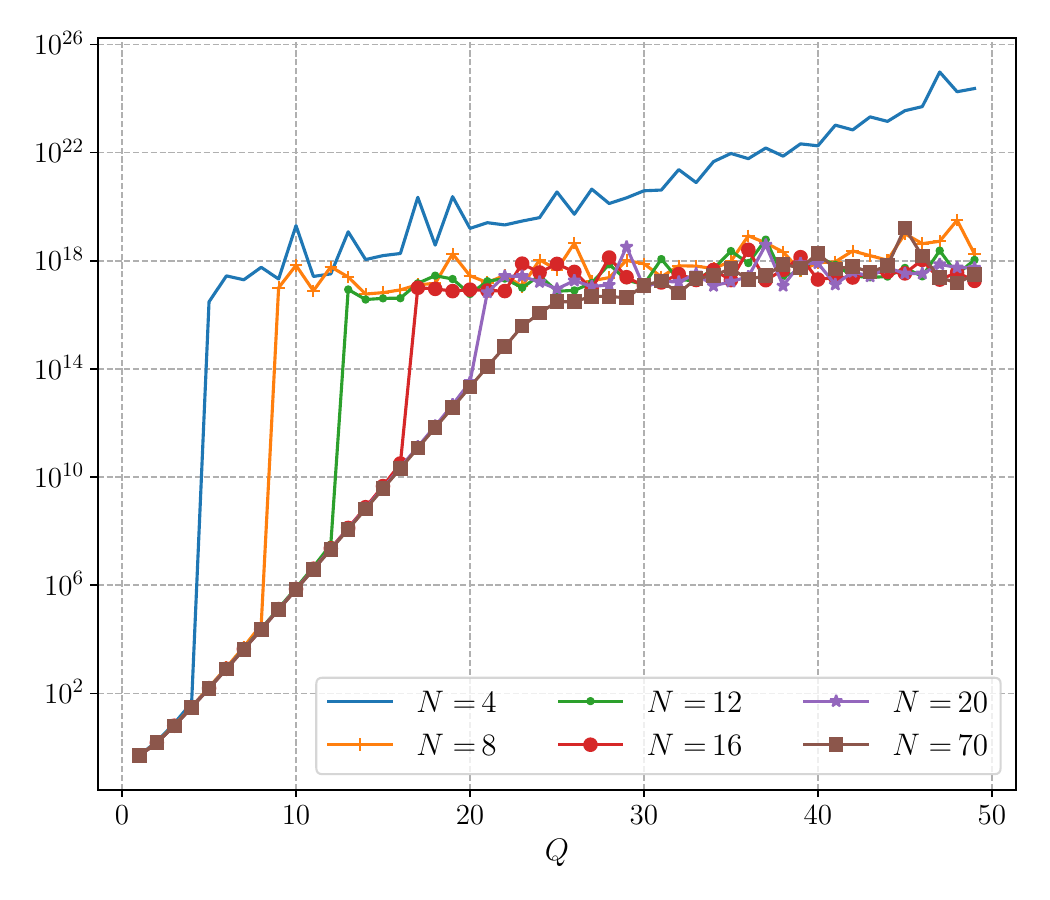}
    \caption{Spectral radius of the matrix $M^{-1}$ as a function of $N$ (Left), and as a function of $Q$ (Right) for Chebyshev 2\textsuperscript{nd} kind polynomials.}
    \label{fig:SpectralRad}
  \end{figure}
\end{remark}

\section{Numerical examples and applications}\label{sec:NumExp}
In this section, we conduct numerical experiments to illustrate the properties of the proposed method. The initial test focuses on evaluating the accuracy of the moment-preserving approach. Subsequently, in Tests 2 to 4, we employ the conservative method within a PDE framework, considering scenarios involving both bounded and unbounded domains. The ensuing analysis includes the assessment of the numerical scheme's accuracy, its conservation properties, and an exploration of the long-time behavior of solutions. In the following, all integrals are approximated using Gaussian quadrature rules associated to the chosen polynomial basis. The number of quadrature points is set to $80$.

\subsection{Test 1: Approximation of functions}
\paragraph{Bounded domain: $[-1,1]$}
Let us consider the approximation of the oscillating and non symetric function
\be\label{testfunctionBounded}
  f(x) = \sin(2 \pi x) + x^2 \cos(2 \pi x),\quad x\in[-1,1].
\ee
We first consider a standard approximation using Legendre polynomials. 
\begin{figure}
    \centering
    \includegraphics[width=.6\linewidth]{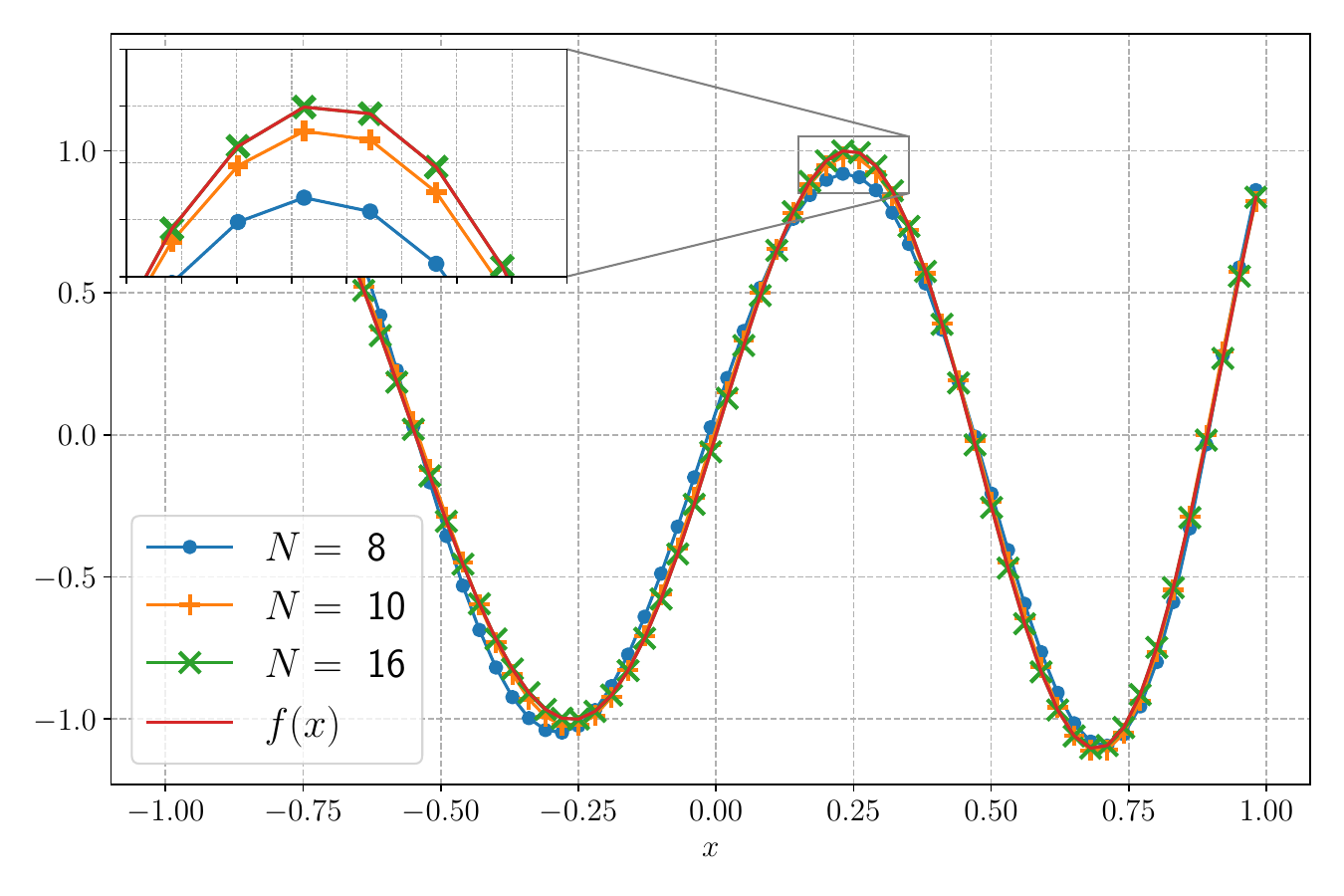}
    \caption{Constrained Legendre approximation for \eqref{testfunctionBounded}.}
    \label{fig:PlotLegendre}
\end{figure}
Figure \ref{fig:PlotLegendre} illustrates the approximate solutions and Table \ref{tab:AccLegendre} shows that the standard method is indeed spectrally accurate. In addition, we observe that moments are indeed preserved without constraining the projection as predicted by \eqref{eq:MomProjError_weight}. Note that errors are sometimes extremely close to $0$ but not exactly $0$ because the Gaussian quadrature used may introduces small machine epsilons.
\begin{table}
  \begin{center}
  \begin{tabular}{c|c|c|c|c|c}
    N & $\|f-f_N\|_2$ & $|m_0-m_{0,N}|$ & $|m_1-m_{1,N}|$ & $|m_2-m_{2,N}|$ & $|m_3-m_{3,N}|$ \\
    \hline
    8 & 9.004e-03 & 2.429e-15 & 2.387e-15 & 2.387e-15 & 2.331e-15 \\
    \hline
    16 & 2.197e-05 & 1.388e-17 & 0 & 0 & 0 \\
    \hline
    32 & 1.050e-13 & 8.327e-17 & 1.110e-16 & 0 & 0\\
  \end{tabular}
  \end{center}
  \caption{Error Analysis: Legendre approximation for \eqref{testfunctionBounded}.}
  \label{tab:AccLegendre}
\end{table}%

Let us now compare the standard and constrained approximations for three polynomial families: Chebyshev 1\textsuperscript{st} kind, Chebyshev 2\textsuperscript{nd} kind, and general Jacobi with parameters $\alpha=1$ and $\beta=-1/2$. The associated outcomes are illustrated in Figure \ref{fig:PlotPoly} and errors are detailed in Tables \ref{tab:AccCheby1}, \ref{tab:AccCheby2}, and \ref{tab:AccJacobi}. Within the Jacobi family, the deliberate choice of $\alpha$ and $\beta$ serves the purpose of disrupting the basis symmetry to evaluate the method's robustness in this context. Across all three cases, it is clear that standard approximations exhibit spectral accuracy, with moments converging in a spectral manner. Consequently, our analysis falls within the range of Theorem \ref{thm:SpectralAccuracyfNc}, affirming spectral accuracy for constrained approximations and exact conservation of the first $M=4$ moments. This property shows clearly in the Tables mentioned earlier.

It is noteworthy that the constraining matrix $\hat{C}_k$ may exhibit pronounced ill-conditioning for large values of $Q$, potentially introducing numerical artifacts. This behaviour is illustrated in Figure \ref{fig:ConditionNumbers} where we plot the condition number of the matrix $M$ appearing in the proof of Theorem \ref{thm:SpectralAccuracyfNc} as well as the minimum over $k=1,\dots,N$ of the condition number of the matrices $\hat{C}_k$. In both cases, we observe large condition numbers as well as their increase with the number of constrained moments. In addition, we want to point out that enough modes should be considered when applying the constrained method. Indeed, a significant increase in the condition number occurs once the number of constraints surpasses the number of modes. Preconditioners may therefore be required to improve numerical stability for large numbers of constrained moments.

\begin{figure}
    \centering
    \includegraphics[width=.45\linewidth]{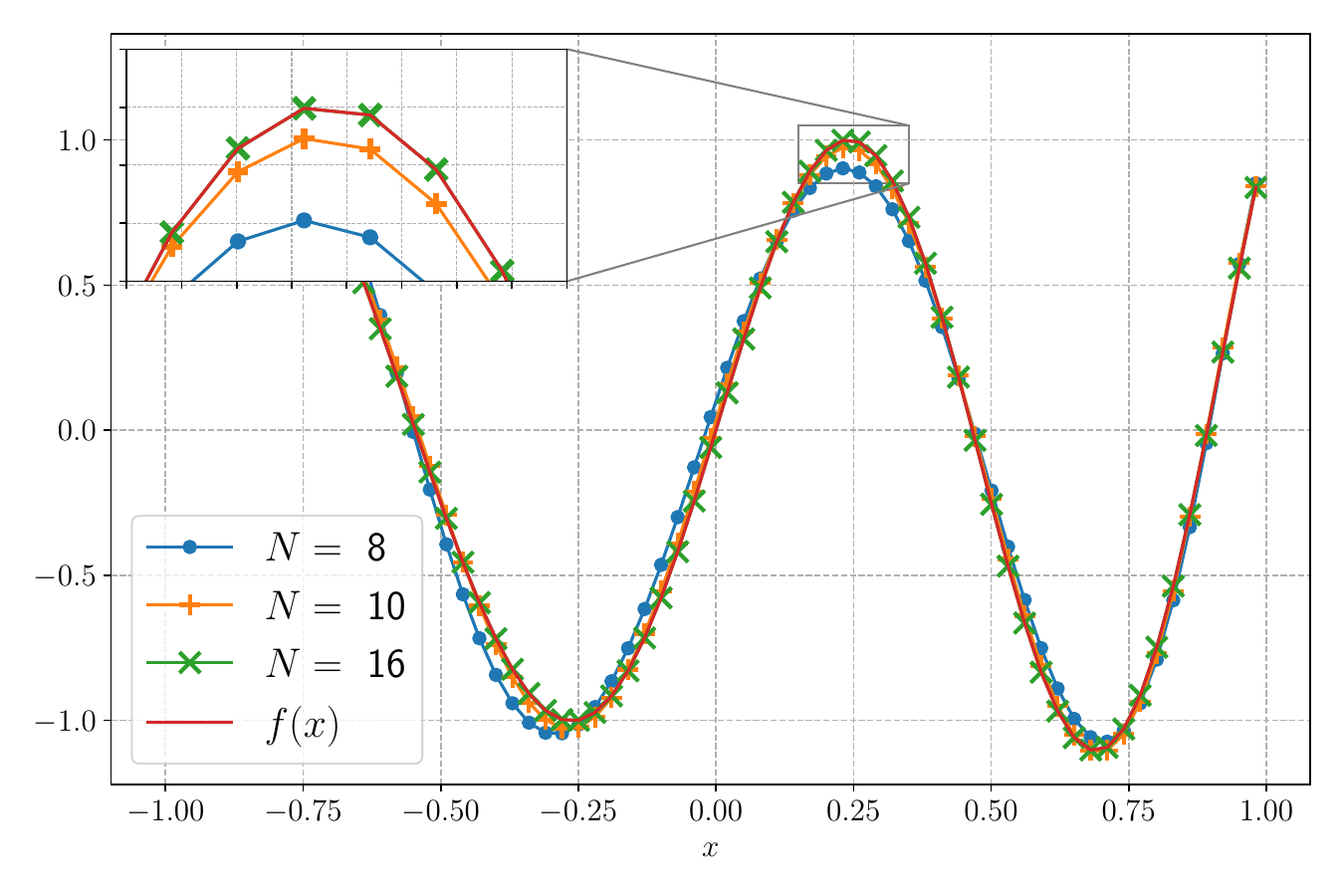}
    \includegraphics[width=.45\linewidth]{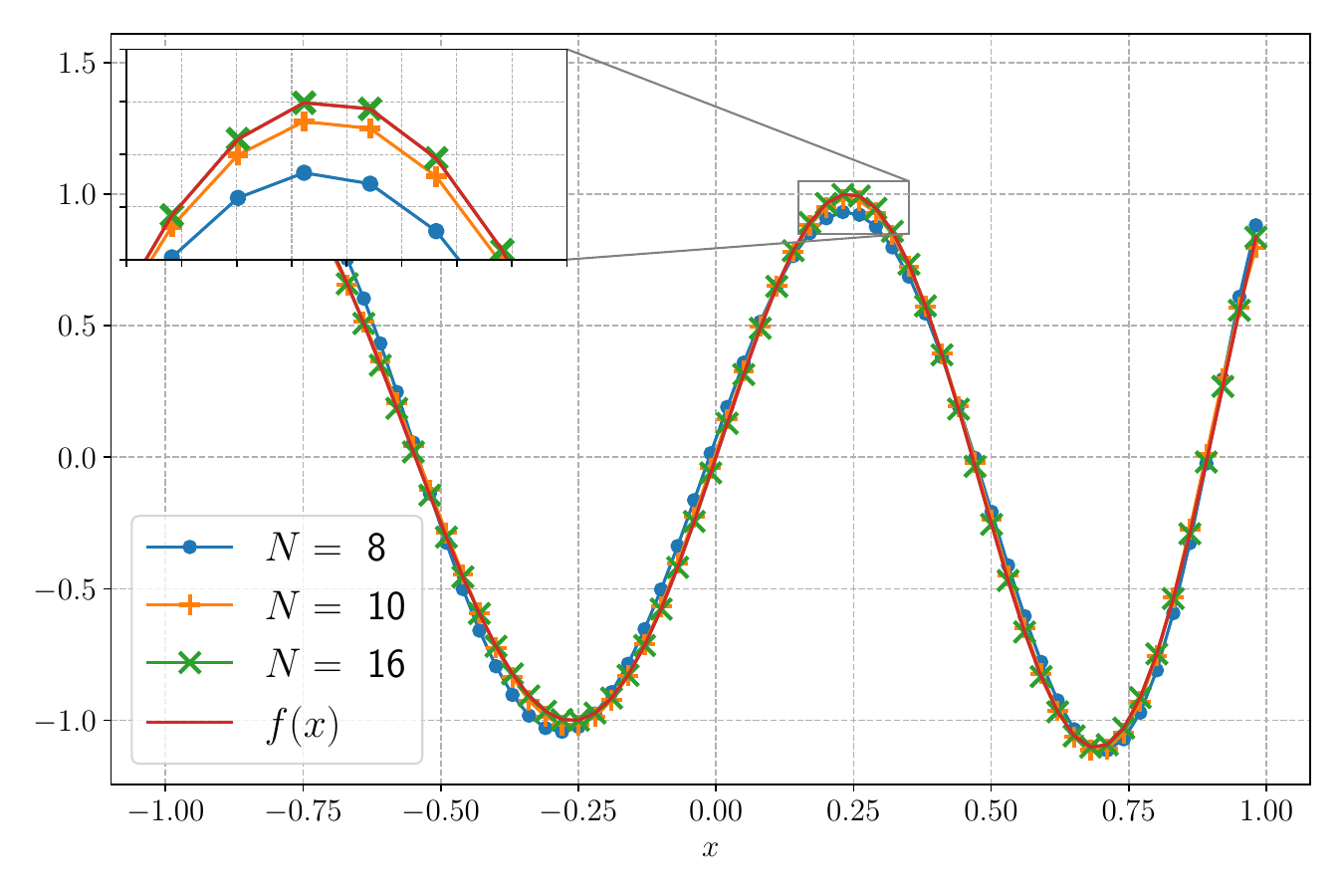}\\
    \includegraphics[width=.45\linewidth]{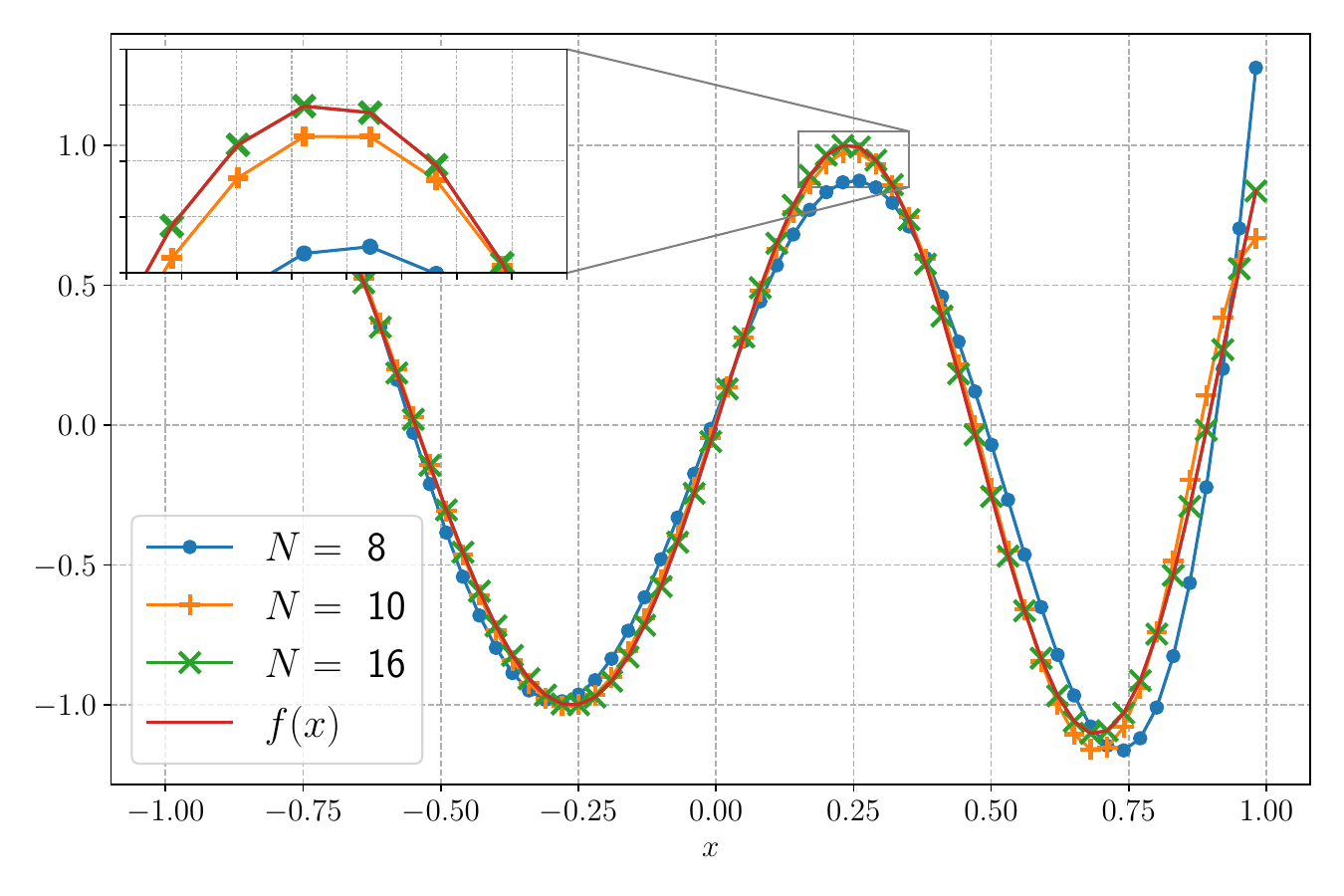}
    \caption{Constrained approximations for Chebyshev 1\textsuperscript{st} kind  (Top Left),  Chebyshev 2\textsuperscript{nd} kind (Top Right) and Jacobi $\alpha=1$, $\beta=-1/2$ (Bottom) for \eqref{testfunctionBounded}.}
    \label{fig:PlotPoly}
\end{figure}

\begin{table}
  \begin{center}
  \begin{tabular}{c|c|c|c|c|c}
    \multicolumn{6}{c}{Standard}\\
    \hline
    N & $\|f-f_N\|_2$ & $|m_0-m_{0,N}|$ & $|m_1-m_{1,N}|$ & $|m_2-m_{2,N}|$ & $|m_3-m_{3,N}|$ \\
    \hline
    8 & 1.034e-01 & 1.746e-03 & 1.404e-02 & 1.958e-02 & 1.538e-02\\
    \hline
    16 & 2.477e-05 & 1.813e-07 & 6.126e-09 & 1.858e-07 & 6.264e-09\\
    \hline
    32 & 2.094e-14 & 1.388e-17 & 5.551e-17 & 2.776e-17 & 5.551e-17 \\
    \multicolumn{6}{c}{Constrained}\\
    \hline
    N & $\|f-f_N^c\|_2$ & $|m_0-m_{0,N}^c|$ & $|m_1-m_{1,N}^c|$ & $|m_2-m_{2,N}^c|$ & $|m_3-m_{3,N}^c|$ \\
    \hline
    8 & 1.039e-01 & 5.551e-17 & 5.551e-17 & 0. & 5.551e-17\\
    \hline
    16 & 2.477e-05 & 0. & 0. & 2.776e-17 & 0.\\
    \hline
    32 & 2.093e-14 & 2.776e-17 & 5.551e-17 & 2.776e-17 & 5.551e-17 \\
  \end{tabular}
  \end{center}
  \caption{Error Analysis: Standard vs. Conservative with Chebyshev 1\textsuperscript{st} kind approximation of \eqref{testfunctionBounded}.}
  \label{tab:AccCheby1}
\end{table}%

\begin{table}
  \begin{center}
  \begin{tabular}{c|c|c|c|c|c}
    \multicolumn{6}{c}{Standard}\\
    \hline
    N & $\|f-f_N\|_2$ & $|m_0-m_{0,N}|$ & $|m_1-m_{1,N}|$ & $|m_2-m_{2,N}|$ & $|m_3-m_{3,N}|$ \\
    \hline
    8 & 1.081e-01 & 8.262e-03 & 5.914e-03 & 8.499e-03 & 6.049e-03\\
    \hline
    16 &  2.701e-05 & 1.388e-06 & 4.889e-08 & 1.398e-06 & 4.921e-08\\
    \hline
    32 & 1.132e-14 & 8.327e-17 & 0 & 2.776e-17 & 5.551e-17\\
    \multicolumn{6}{c}{Constrained}\\
    \hline
    N & $\|f-f_N^c\|_2$ & $|m_0-m_{0,N}^c|$ & $|m_1-m_{1,N}^c|$ & $|m_2-m_{2,N}^c|$ & $|m_3-m_{3,N}^c|$ \\
    \hline
    8 & 1.029e-01 & 1.388e-17 & 0 & 0 & 0\\
    \hline
    16 & 2.629e-05 & 8.327e-17 & 0 & 2.776e-17 & 5.551e-17\\
    \hline
    32 & 1.135e-14 & 1.388e-17 & 0 & 0 & 5.551e-17\\
  \end{tabular}
  \end{center}
  \caption{Error Analysis: Standard vs. Conservative with Chebyshev 2\textsuperscript{nd} kind approximation of \eqref{testfunctionBounded}.}
  \label{tab:AccCheby2}
\end{table}%

\begin{table}
  \begin{center}
  \begin{tabular}{c|c|c|c|c|c}
    \multicolumn{6}{c}{Standard}\\
    \hline
    N & $\|f-f_N\|_2$ & $|m_0-m_{0,N}|$ & $|m_1-m_{1,N}|$ & $|m_2-m_{2,N}|$ & $|m_3-m_{3,N}|$ \\
    \hline
    8 & 2.738e-01 & 4.374e-02 & 4.507e-02 & 0.437e-02 &  4.514e-02\\
    \hline
    16 & 8.873e-05 & 7.374e-06 & 7.433e-06 & 7.373e-06 & 7.434e-06 \\
    \hline
    32 & 2.237e-12 & 8.263e-14 & 9.082e-14 & 8.271e-14 & 9.093e-14 \\
    \multicolumn{6}{c}{Constrained}\\
    \hline
    N & $\|f-f_N^c\|_2$ & $|m_0-m_{0,N}^c|$ & $|m_0-m_{1,N}^c|$ & $|m_2-m_{2,N}^c|$ & $|m_3-m_{3,N}^c|$ \\
    \hline
    8 & 1.787e-01 & 6.939e-17 & 0 & 2.776e-17 & 5.551e-17\\
    \hline
    16 & 6.695e-05 & 1.110e-16 & 0 & 1.110e-16 & 5.551e-17\\
    \hline
    32 & 1.868e-12 & 1.388e-16 & 5.551e-17 & 5.551e-17 & 0\\
  \end{tabular}
  \end{center}
  \caption{Error Analysis: Standard vs. Conservative with Jacobi ($\alpha=1$, $\beta=-\frac{1}{2}$) approximation of \eqref{testfunctionBounded}.}
  \label{tab:AccJacobi}
\end{table}%

\begin{figure}[!ht]
  \centering
  \includegraphics[width=.45\linewidth]{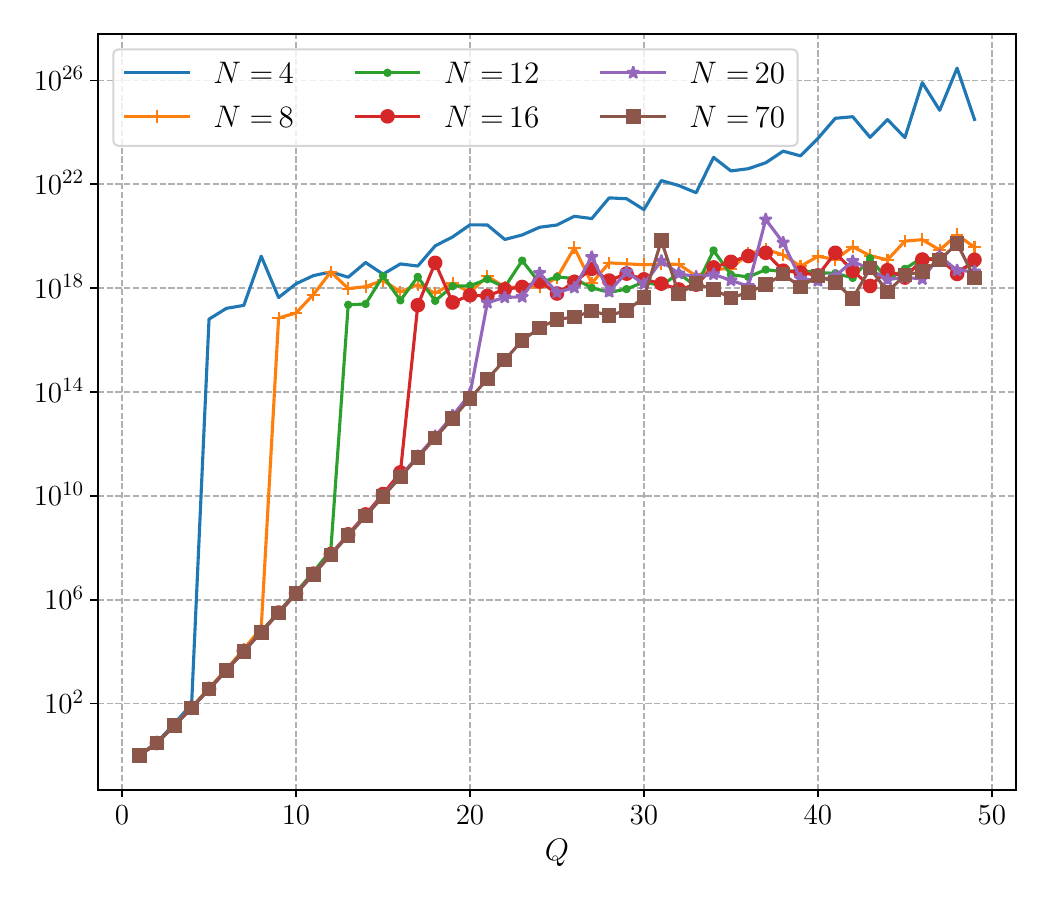}
  \includegraphics[width=.45\linewidth]{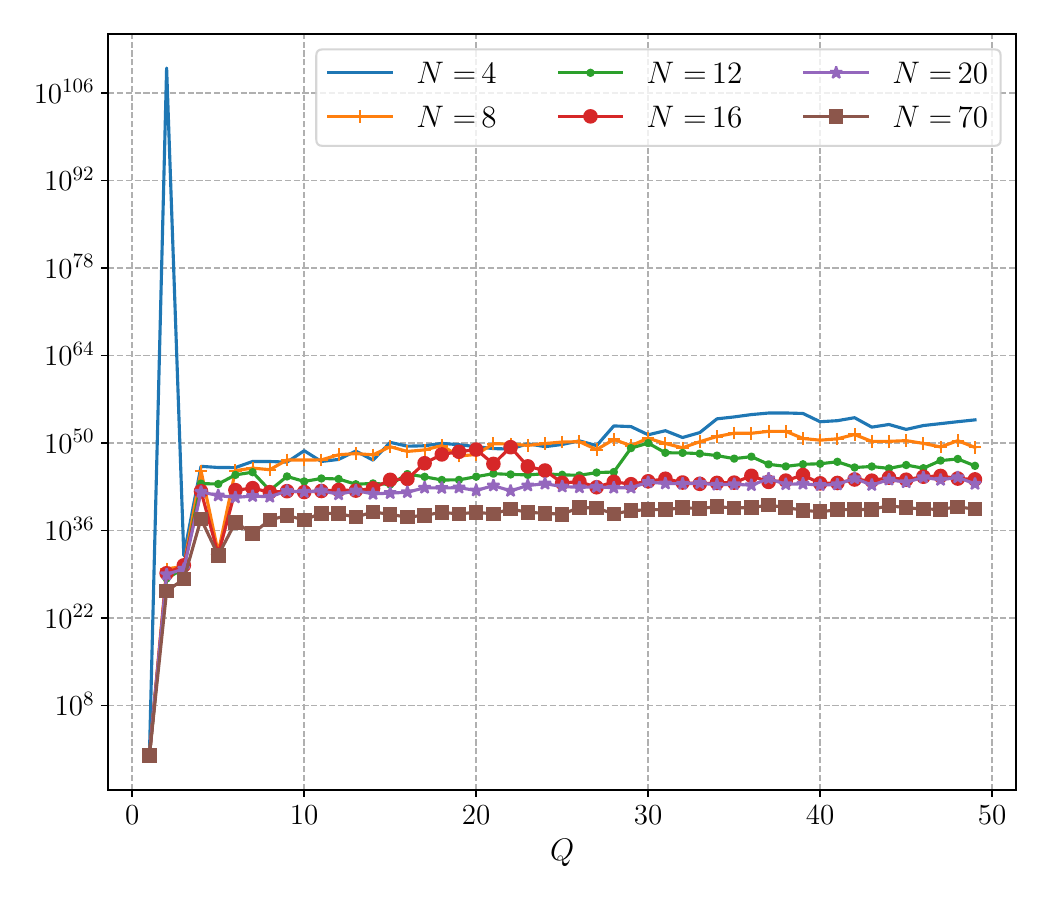}
  \caption{ Condition number of the matrix M (Left), and minimum of the condition number of the matrices $\hat{C}_k$ (Right) as functions of $Q$ for Chebyshev 2\textsuperscript{nd} kind polynomials}.
  \label{fig:ConditionNumbers}
\end{figure}

\paragraph{Unbounded domain: $\mathbb{R}$}
Continuing, we extend our approximation to $\mathbb{R}$ utilizing Hermite functions. Let $H_k$ denote the $k$\textsuperscript{th} Hermite polynomial. We define the  $k$\textsuperscript{th} symmetrically weighted Hermite functions as 
\begin{equation}
    \psi_k(x)=H_k(x)e^{-x^2/2}.
\end{equation} 
Specifically, we want to approximate the difference between two Gaussian functions deliberately crafting a non-symmetric outcome:
\be\label{testfunctionHermite}
  f(x) = \frac{3}{\sqrt{2\pi}}e^{(x+3)/2}-\frac{1}{\sqrt{\pi}}e^{(x-2)},\quad x\in\mathbb{R}.
\ee

\begin{figure}
    \centering
    \includegraphics[width=.6\linewidth]{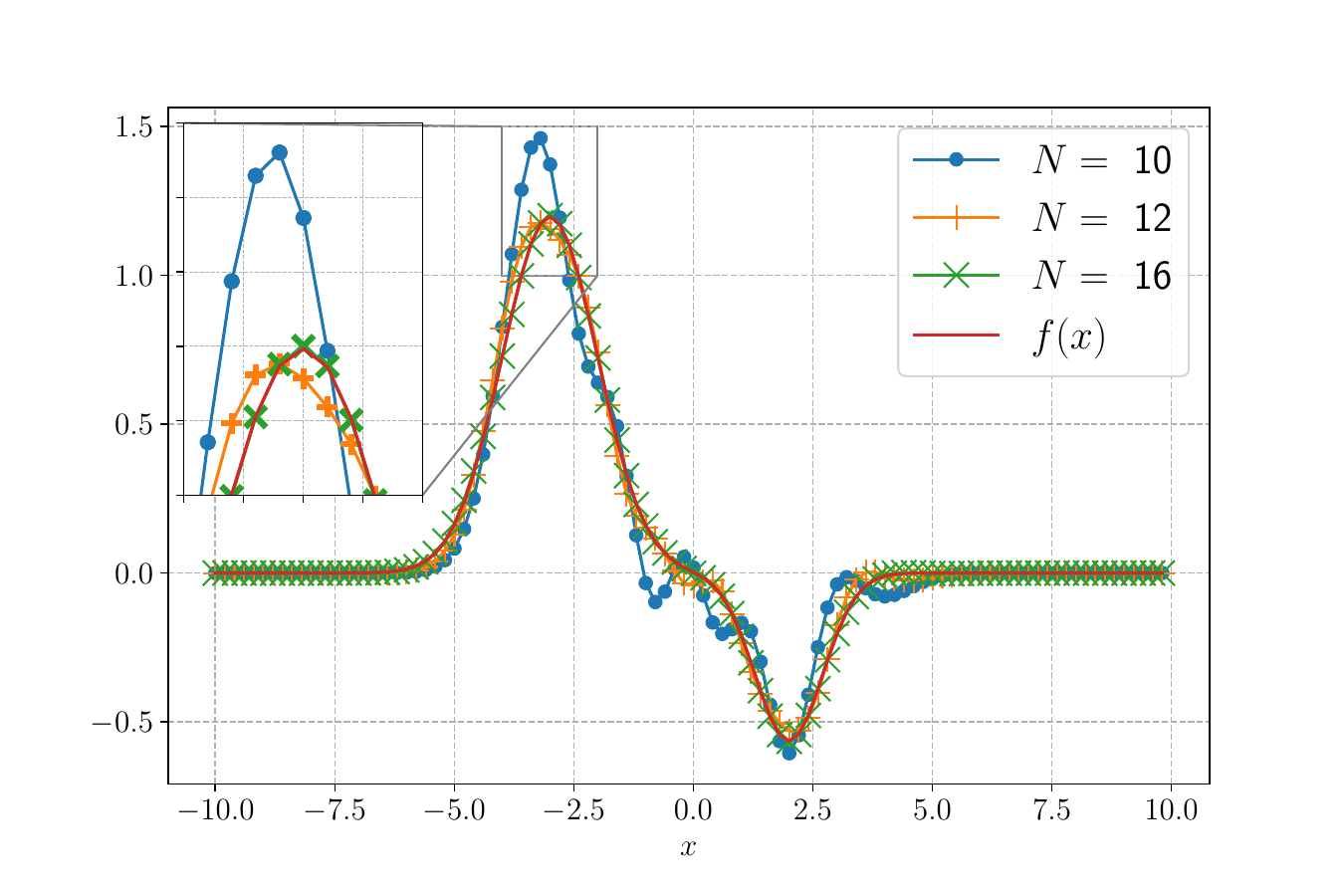}
    \caption{Constrained Hermite approximation for \eqref{testfunctionHermite}.}
    \label{fig:PlotHermite}
\end{figure}

We illustrate the results in Figure \ref{fig:PlotHermite} and, again, the results (See Table \ref{tab:AccHermite}) on the standard approximation ensures that we are within the scope of Theorem \ref{thm:SpectralAccuracyfNc}. As predicted, the conservative method exhibits spectral accuracy, along with conservation of moments.

\begin{table}
  \begin{center}
  \begin{tabular}{c|c|c|c|c|c}
    \multicolumn{6}{c}{Standard}\\
    \hline
    N & $\|f-f_N\|_2$ & $|m_0-m_{0,N}|$ & $|m_1-m_{1,N}|$ & $|m_2-m_{2,N}|$ & $|m_3-m_{3,N}|$ \\
    \hline
    8 & 1.264e-02 & 0.583e-01 &  1.718e00 & 1.130e+01 & 3.686e+01\\
    \hline
    16 & 1.716e-03 & 6.692e-03 & 2.060e-02 & 2.293e-01 & 7.457e-01 \\
    \hline
    32 & 4.995e-09 & 1.698e-09 & 5.175e-09 & 9.470e-08 & 3.552e-07 \\
    \multicolumn{6}{c}{Constrained}\\
    \hline
    N & $\|f-f_N^c\|_2$ & $|m_0-m_{0,N}^c|$ & $|m_1-m_{1,N}^c|$ & $|m_2-m_{2,N}^c|$ & $|m_3-m_{3,N}^c|$ \\
    \hline
    8 & 4.324e-01 & 0 & 1.776e-15 & 3.553e-15 & 2.842e-14\\
    \hline
    16 & 2.879e-03 & 2.220e-16 & 1.776e-15 & 0 & 0\\
    \hline
    32 & 5.055e-09 & 2.220e-16 & 1.776e-15 & 0 & 0 \\
  \end{tabular}
  \end{center}
  \caption{Error Analysis: Standard vs. Conservative Hermite function approximation for \eqref{testfunctionHermite}.}
  \label{tab:AccHermite}
\end{table}%

\paragraph{Unbounded domain: $\mathbb{R}^+$}
We now look at the approximation on $\mathbb{R}^+$ using symmetric Laguerre functions. We define for a Laguerre polynomial $\mathcal{L}_k$ the $k$\textsuperscript{th} Laguerre function as
\be
    \xi_k(x)=\mathcal{L}_k(x)e^{-x/2}.
\ee
Let us consider the function
\be\label{testfunctionLaguerre}
  f(x) = (x^3-2x+\sin(x))e^{-x},\quad x\in\mathbb{R}^+.
\ee

\begin{figure}
    \centering
    \includegraphics[width=.6\linewidth]{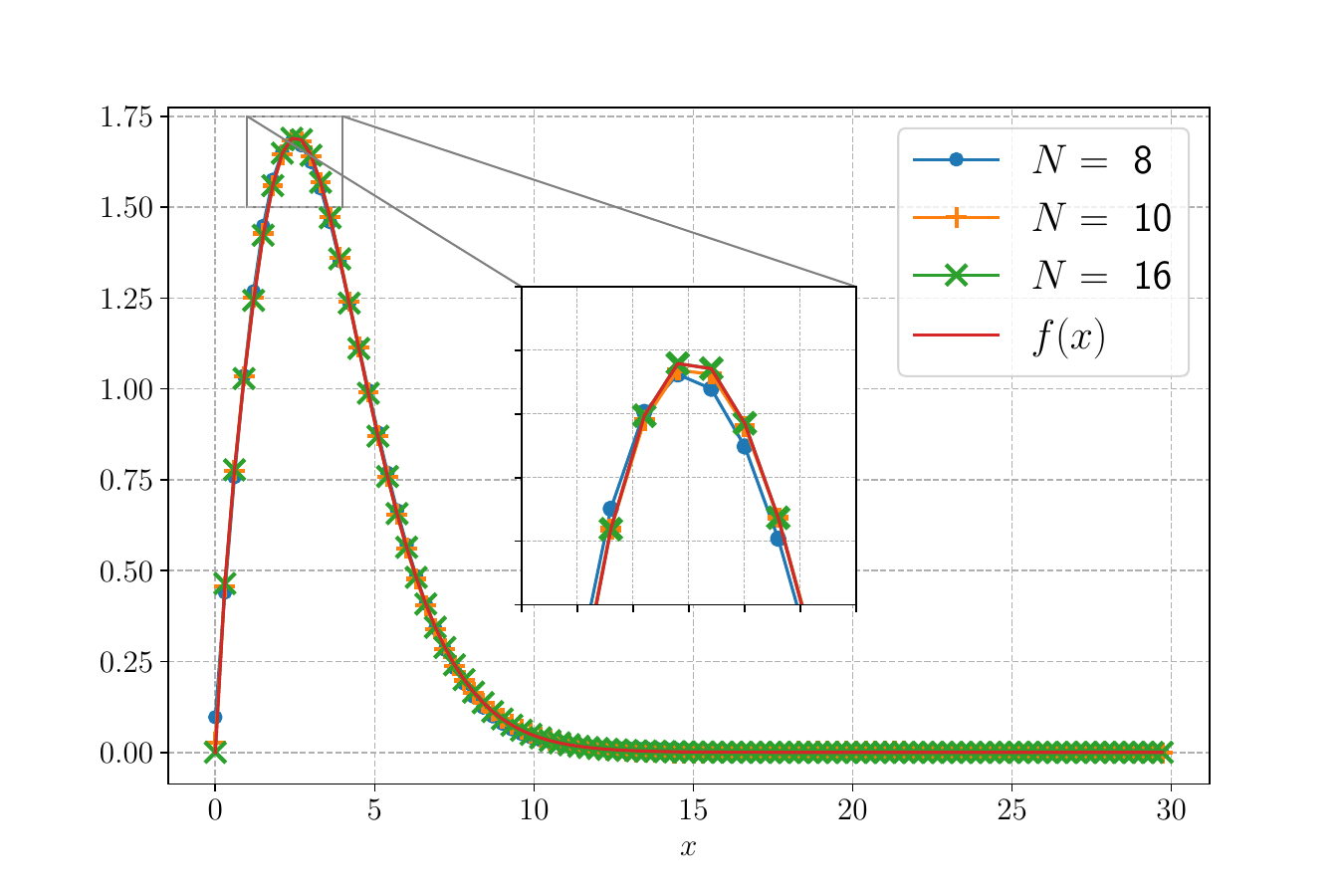}
    \caption{Constrained Laguerre approximation for \eqref{testfunctionLaguerre}.}
    \label{fig:PLotLaguerre}
\end{figure}

The approximation is illustrated in Figure \ref{fig:PLotLaguerre} and, consistent with prior cases, the observations in Table \ref{tab:AccLaguerre} shows the spectral accuracy of the conservative method. Although errors on the moments exhibit a slight increase compared to previous cases, it is noteworthy that this discrepancy can be attributed to numerical artifacts arising from the ill-conditioning of the constraining matrix $\hat{C}_k$ and numerical integration intricacies. 
\begin{table}
  \begin{center}
  \begin{tabular}{c|c|c|c|c|c}
    \multicolumn{6}{c}{Standard}\\
    \hline
    N & $\|f-f_N\|_2$ & $|m_0-m_{0,N}|$ & $|m_1-m_{1,N}|$ & $|m_2-m_{2,N}|$ & $|m_3-m_{3,N}|$ \\
    \hline
    8 & 3.761e-02 & 3.780e-02 & 1.225e+00 & 3.968e+01 & 1.295e+03\\
    \hline
    16 & 9.552e-05 & 1.227e-04 & 7.954e-03 & 5.157e-01 & 3.350e+01\\
    \hline
    32 & 2.337e-11 & 3.174e-11 & 4.092e-09 & 5.276e-07 & 6.807e-05 \\
    \multicolumn{6}{c}{Constrained}\\
    \hline
    N & $\|f-f_N^c\|_2$ & $|m_0-m_{0,N}^c|$ & $|m_1-m_{1,N}^c|$ & $|m_2-m_{2,N}^c|$ & $|m_3-m_{3,N}^c|$ \\
    \hline
    8 & 4.834e-02 & 8.882e-16 & 3.553e-15 & 3.126e-13 & 3.638e-12 \\
    \hline
    16 & 1.174e-04 & 8.882e-16 & 0 & 5.684e-14 & 9.095e-13\\
    \hline
    32 & 2.622e-11 & 0 & 3.553e-15 & 2.842e-14 & 2.274e-13 \\
  \end{tabular}
  \end{center}
  \caption{Error Analysis: Standard vs. Conservative using Laguerre Function approximation for \eqref{testfunctionLaguerre}.}
  \label{tab:AccLaguerre}
\end{table}%

 To conclude this first test let us again consider Laguerre functions but for the approximation of
 \begin{equation}\label{testfunctionLaguerreBad}
  f(x) = \frac{1}{\sqrt{2\pi\sigma}x}\exp\left(-\frac{(\ln(x)-\mu)^2}{2\sigma}\right),\quad x\in\mathbb{R}^{+,*}.
 \end{equation}
where $\sigma=0.2$ and $\mu=\ln(40)-0.2$. This function is badly approximated by negative exponentials at infinity, yet it holds significance for Test 4. As depicted in Figure \ref{fig:ConvLaguerreBad}, spectral accuracy is not attained neither on the standard approximation of \eqref{testfunctionLaguerreBad} nor on its moments. Consequently, this experiment falls outside the assumptions of Theorem \ref{thm:SpectralAccuracyfNc}, since one lacks both spectral accuracy properties needed. Nevertheless, it is noteworthy that the resultant conservative approximation successfully preserves the $0$\textsuperscript{th} order moment and achieves fifth-order accuracy. 

\begin{figure}
    \centering
    \includegraphics[width=.45\linewidth]{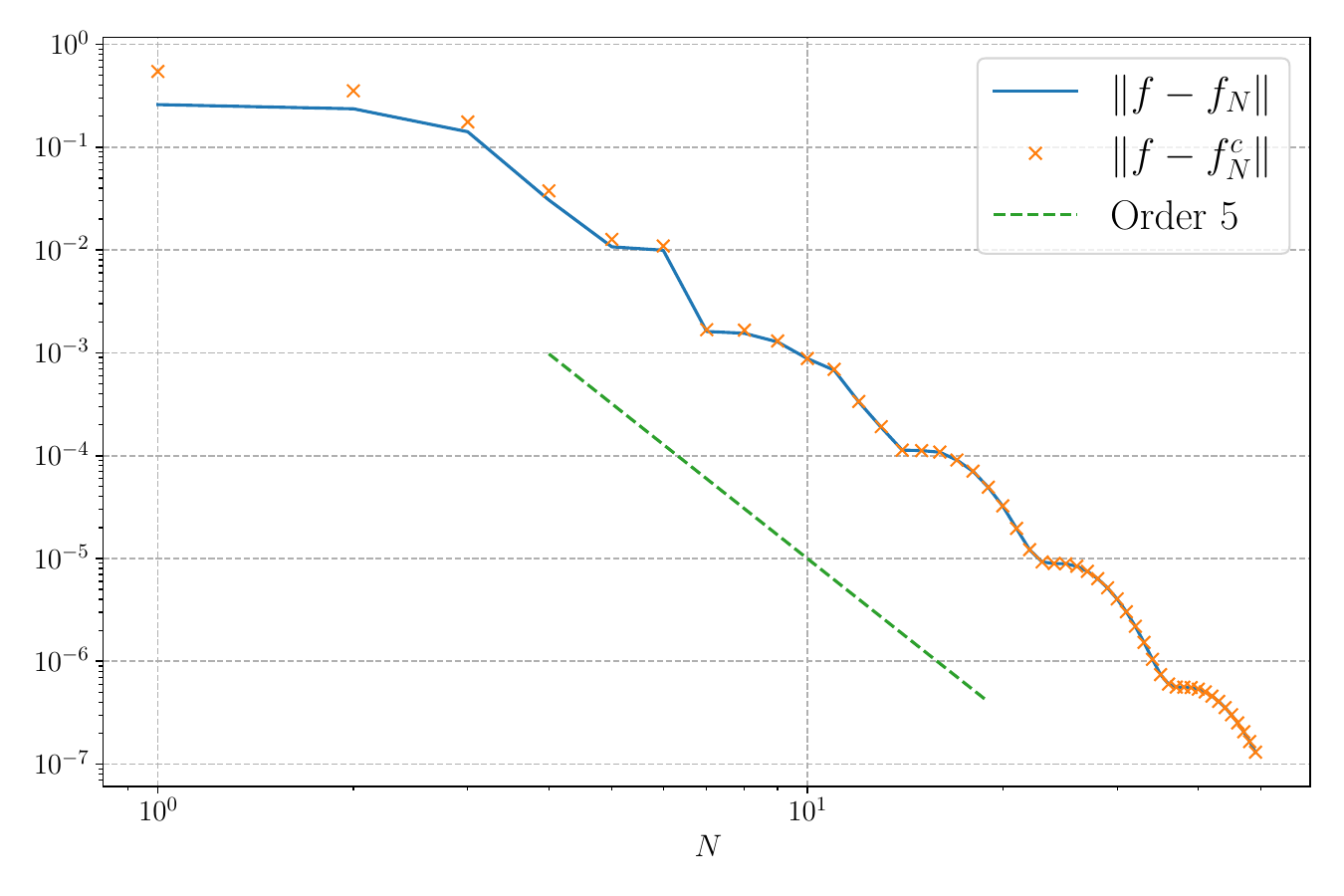}
    \includegraphics[width=.45\linewidth]{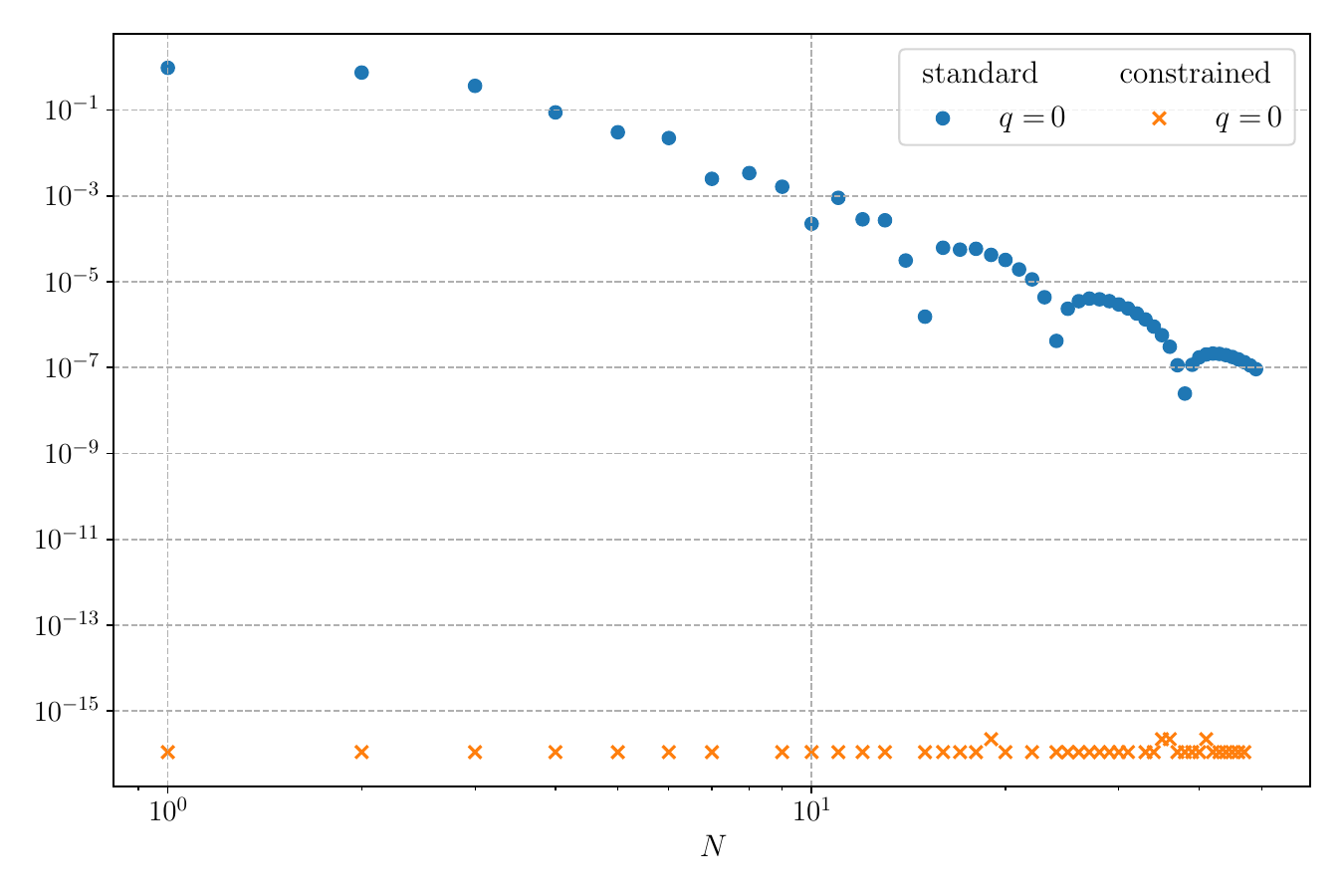}
    \caption{Error Analysis: Solution and first Moment using Laguerre functions for \eqref{testfunctionLaguerreBad}.}
    \label{fig:ConvLaguerreBad}
\end{figure}

In the next tests, we apply the conservative method in a PDE setting to various Fokker-Planck equations.

\subsection{Test 2: Kinetic Fokker-Planck equation}
Let us now consider the following kinetic Fokker-Planck equation on $\mathbb{R}$:
\be\label{FokkerPLanck}
\left\lbrace\begin{aligned}
  &\frac{\partial}{\partial t}f=\partial_v\left((v-\mu)f+T\partial_v f\right):=L_{FP}(f),\\
  &f(0,v) = f^0(v).
\end{aligned}\right.
\ee
In this case, the unknown $f(t,v)$ describes the amount of particles moving with velocity $v$ at time $t$. In this setting, the particle will be forced towards a mean velocity $\mu$ and an average temperature $T$. It is well known that this equation admits a steady state $f_\infty$ given by a Maxwellian distribution:
\begin{equation}\label{FPequilibrium}
  f_\infty=\frac{\rho}{\sqrt{2\pi T}}\exp\left(\frac{-(v-\mu)^2}{2T}\right),
\end{equation}
where $\rho$ is the total mass of particles. In the following, let us consider the initial data
\be
f^0(v)=\frac{1}{\sqrt{\pi}}\left(e^{-(v+2)^2}+e^{-(v-2)^2}\right)
\ee
This distribution admits a mass $\rho^0$, mean velocity $\mu^0$ and temperature $T^0$ given by
\begin{equation}\label{eq:Momentsf0FP}
    \rho^0 = \int_\R f^0(v)\,\mathrm{d}v,\quad\mu^0  = \frac{1}{\rho^0}\int_\R vf^0(v)\,\mathrm{d}v,\quad T^0 = \frac{1}{\rho^0}\int_\R |v-\mu^0|^2f^0(v)\,\mathrm{d}v.
\end{equation}
In addition, we set the parameters of $L_{FP}$ in \eqref{FokkerPLanck} as
\be\label{eq:choiceParamFP}
    \quad\mu=\mu^0 \text{ and} \quad T=T^0.
\ee 
With these choice of parameters, one can in particular expect preservation of mass, mean velocity and temperature.
We now consider the symmetrically weighted Hermite functions: $\psi_k(v)=H_k(v)e^{-v^2/2}$. By definition of the Hermite polynomials, we have the relations:
\be
\begin{aligned}
  H_k &= 2vH_{k-1}-2(k-1)H_{k-2}\\
  H_k' &= 2kH_{k-1}.
\end{aligned}
\ee
Therefore, the Hermite functions $\psi_k$ satisfy:
\be
\begin{aligned}
  \psi_k &= 2v\psi_{k-1}-2(k-1)\psi_{k-2},\\
  \psi_k' &= k\psi_{k-1}-\frac{1}{2}\psi_{k+1},\\
  \psi_k'' &= k(k-1)\psi_{k-2}-\left(k+\frac{1}{2}\right)\psi_{k}+\frac{1}{4}\psi_{k+2}.
\end{aligned}
\ee
Using these relations, one can explicitly compute
\be
\begin{aligned}
  L_{FP}(\psi_k)&=\psi_k + (v-\mu)\psi_k' + T\psi_k''\\
  &= k(k-1)(1+T)\psi_{k-2}-\frac{\mu}{2}\psi_{k-1}+\left(-kT-\frac{(T-1)}{2}\right)\psi_k\\
  &\quad+\frac{\mu}{2}\psi_{k+1}+\frac{1}{4}\left(T-1\right)\psi_{k+2}.
\end{aligned}
\ee
We can now proceed to use a Galerkin method to solve \eqref{FokkerPLanck}. We denote by $\mathcal{P}_N$ and $\mathcal{P}_N^c$ the standard and constrained Hermite function approximation. We want to find $f_N\in S_N$ solution to
\begin{equation}\label{eq:galerkinFP}
    \left\lbrace\begin{aligned}
    &\frac{\partial}{\partial t}f_N = \mathcal{P}_N\left(L_{FP}(f_N)\right),\\
    &f_N(0,v)=\mathcal{P}_N(f^0)(v).
\end{aligned}\right.
\end{equation}
The moment constrained problem writes
\begin{equation}\label{eq:galerkinFPc}
    \left\lbrace\begin{aligned}
    &\frac{\partial}{\partial t}f_N^c = L_{FP,N}^c(f_N^c),\\
    &f_N(0,v)=\mathcal{P}_N^c(f^0)(v),
\end{aligned}\right.
\end{equation}
where $L_{FP,N}^c$ is solution to
\begin{equation}\label{eq:constrOperatorFP}
    L_{FP,N}^c(f) = {\rm argmin} \left\{\| g_N - L_{FP}(f) \|^2_{L^2_\omega}\, :\,g_N\in S_N,\,\, \langle g_N,v^q \rangle = 0,\,q=0,1,2,3\right\}.
\end{equation}
Let us emphasize that \eqref{eq:constrOperatorFP} corresponds to \eqref{def:constrainedApprox} where we ensure the first three moments of $L_{FP}$ to be zero. For the time stepping method, we utilize a classical fourth order Runge-Kutta (RK4) method with fixed time step $\Delta t =10^{-4}$. In Figure \ref{SnapFP} one can observe a very good agreement between the two overlapping solutions as well as a trend towards the steady state $f_\infty$. This observation suggests that the constrained approximation may be able to finely capture long-time properties of the solution.
\begin{figure}
  \centering
   \includegraphics[width=.9\linewidth]{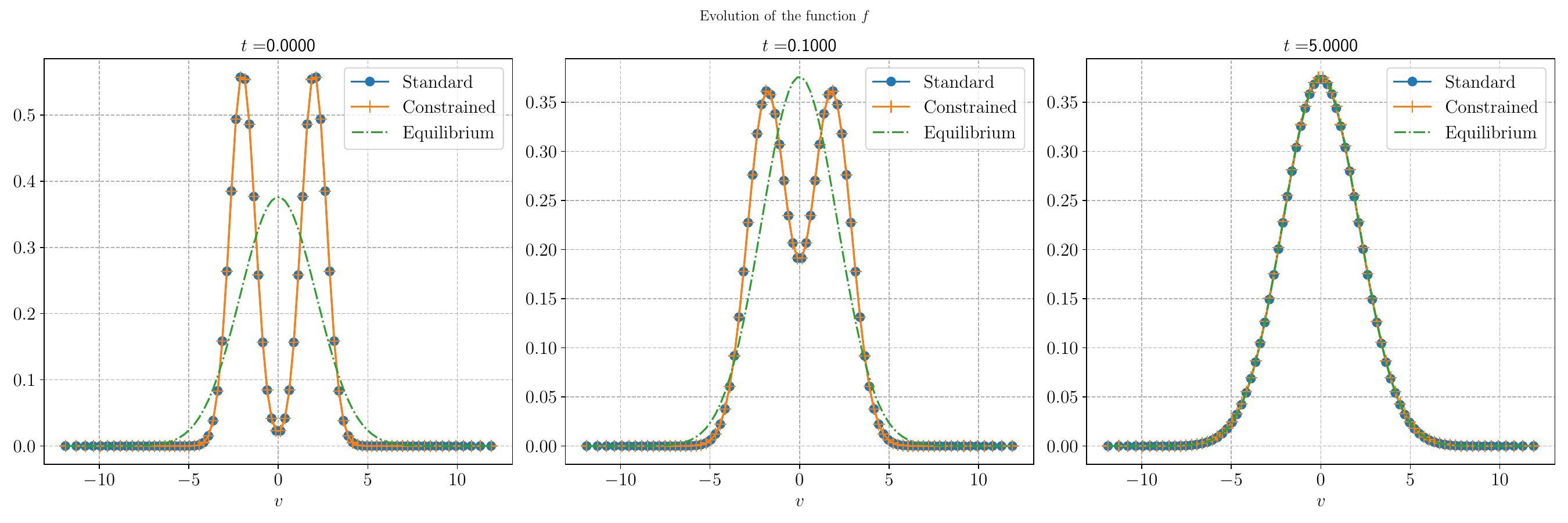}
  \label{SnapFP}
  \caption{Fokker-Planck model: Snapshots of the distribution at $t=0,0.1$ and $5$, $N=32$.}
\end{figure}

We also present in Figure \ref{fig:nonzerobulkV} a test case for the following initial data that has a non-zero mean velocity:
\begin{equation*}
  f^0(v)=\frac{1}{\sqrt{\pi}}\left(e^{-(v+3)^2}+e^{-v^2}\right).
\end{equation*}
This test becomes relevant for spatially inhomogeneous kinetic problem. We observe that our constrained method is still able to nicely capture the solution, even in this more challenging case.
\begin{figure}
  \centering
  \includegraphics[width=.9\linewidth]{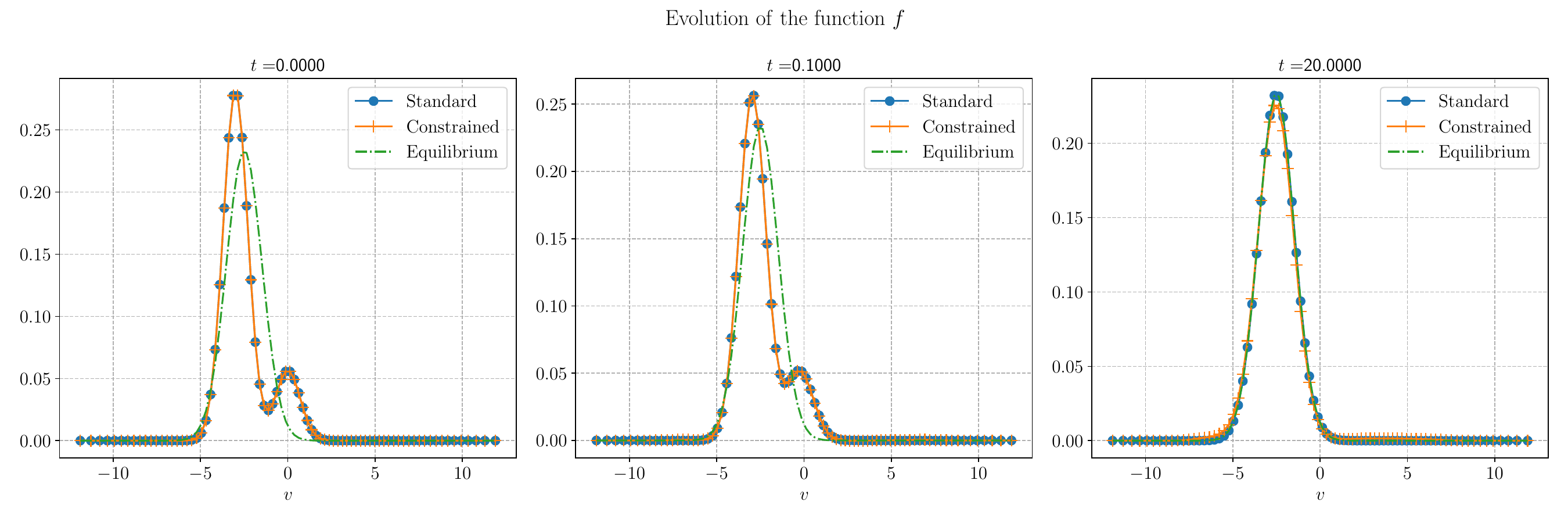}
  \caption{Fokker-Planck model, non-zero mean velocity: Snapshots of the distribution at $t=0,0.1$ and $20$, $N=32$.}
  \label{fig:nonzerobulkV}
\end{figure}

\paragraph{Spectral accuracy and conservations}
Let us now investigate the accuracy of the method. In Figure \ref{ErrMomFokkerPlanck} we set a final time $T_f=0.1$ and compare the solution to a reference one computed with $N=32$ modes. We do the same comparison for the first three moments. A first observation is that the standard method does exhibit spectral accuracy both on the solution and on the first three moments. This therefore falls within the scope of Theorem \eqref{thm:SpectralAccuracyfNc} and the constrained method is indeed spectrally accurate. Note that in this particular case, the mean velocity is $0$, hence the machine accuracy in the bottom left panel of Figure \ref{ErrMomFokkerPlanck}.
\begin{figure}
  \centering
   \includegraphics[width=.48\linewidth]{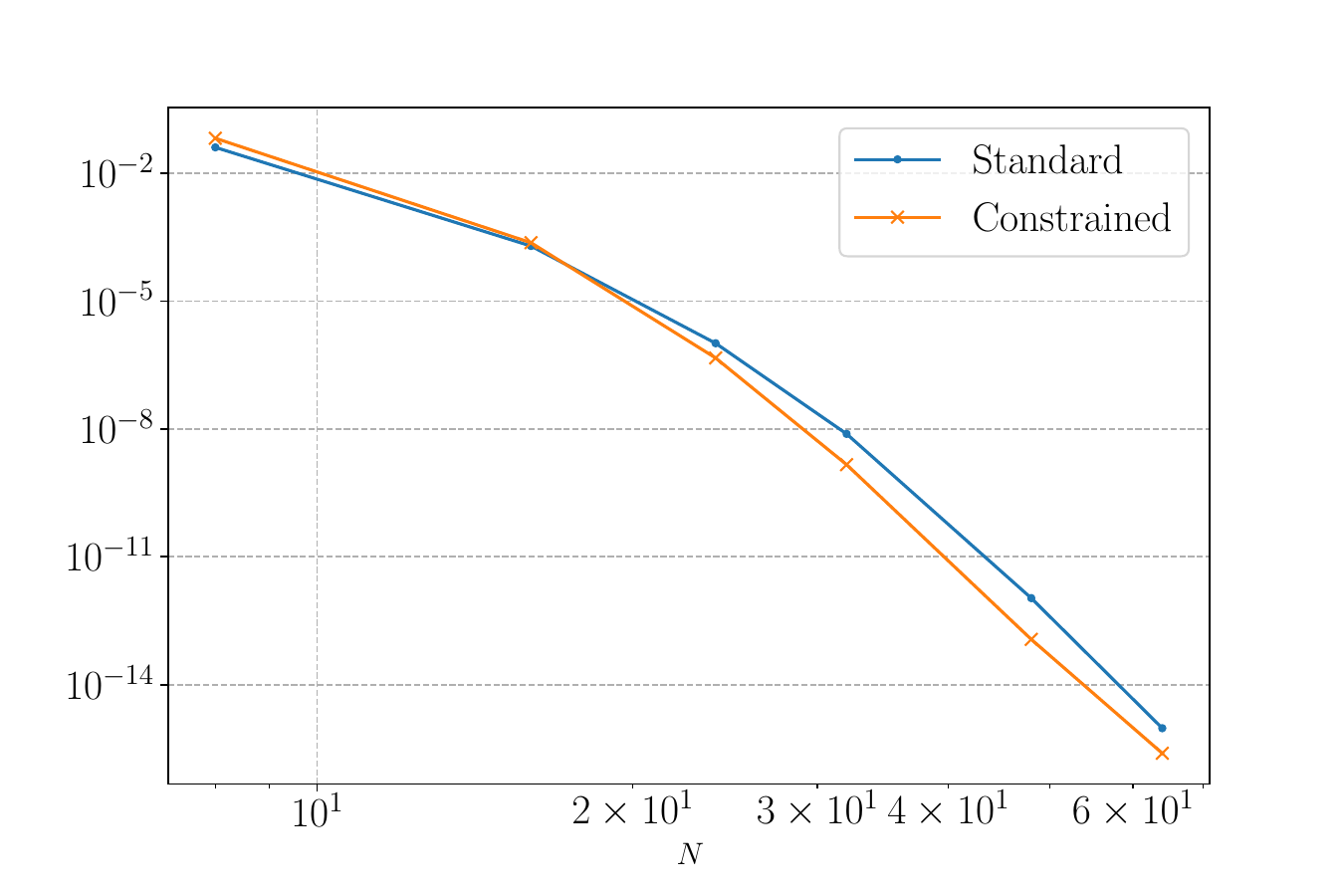}
   \includegraphics[width=.48\linewidth]{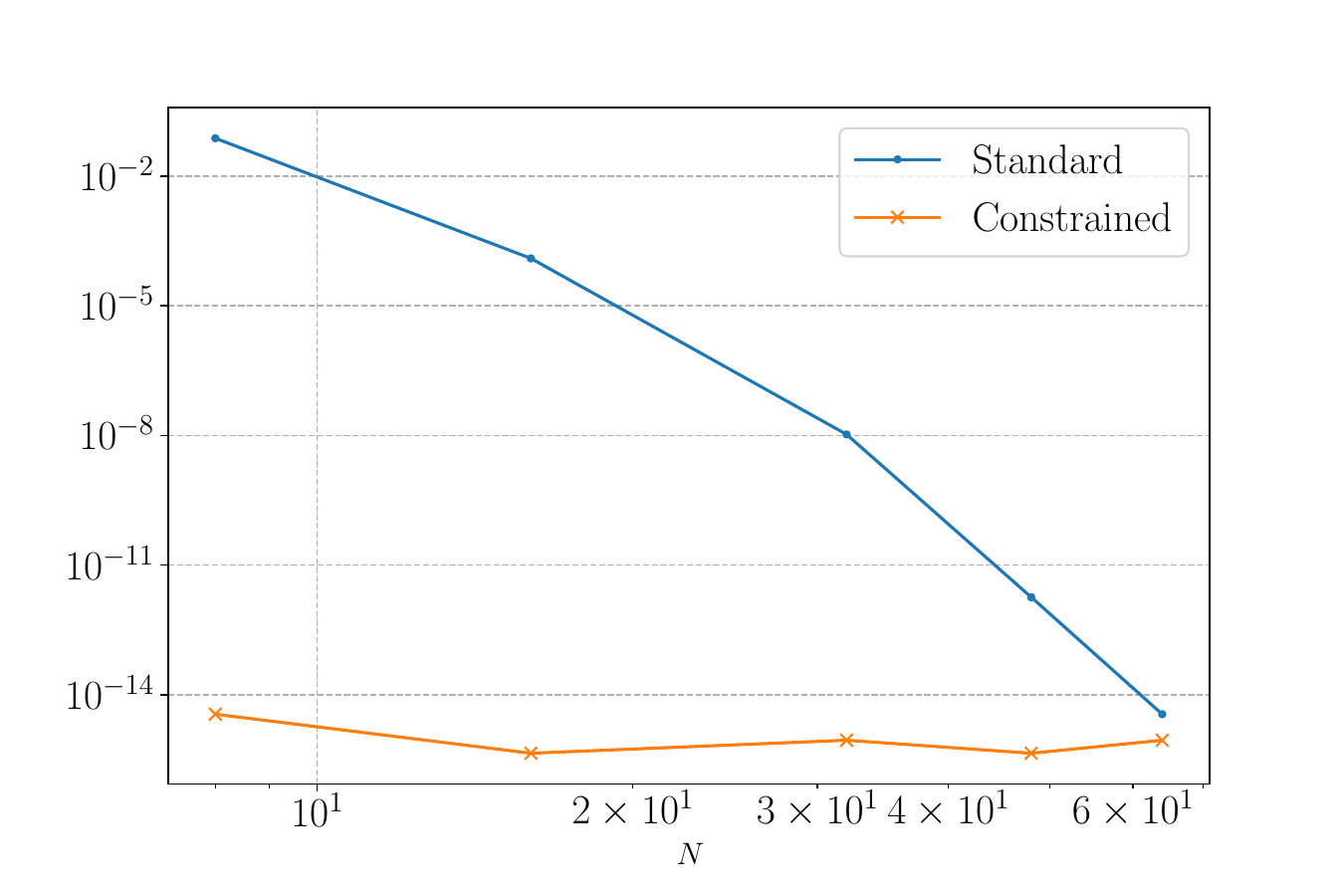}\\
   \includegraphics[width=.48\linewidth]{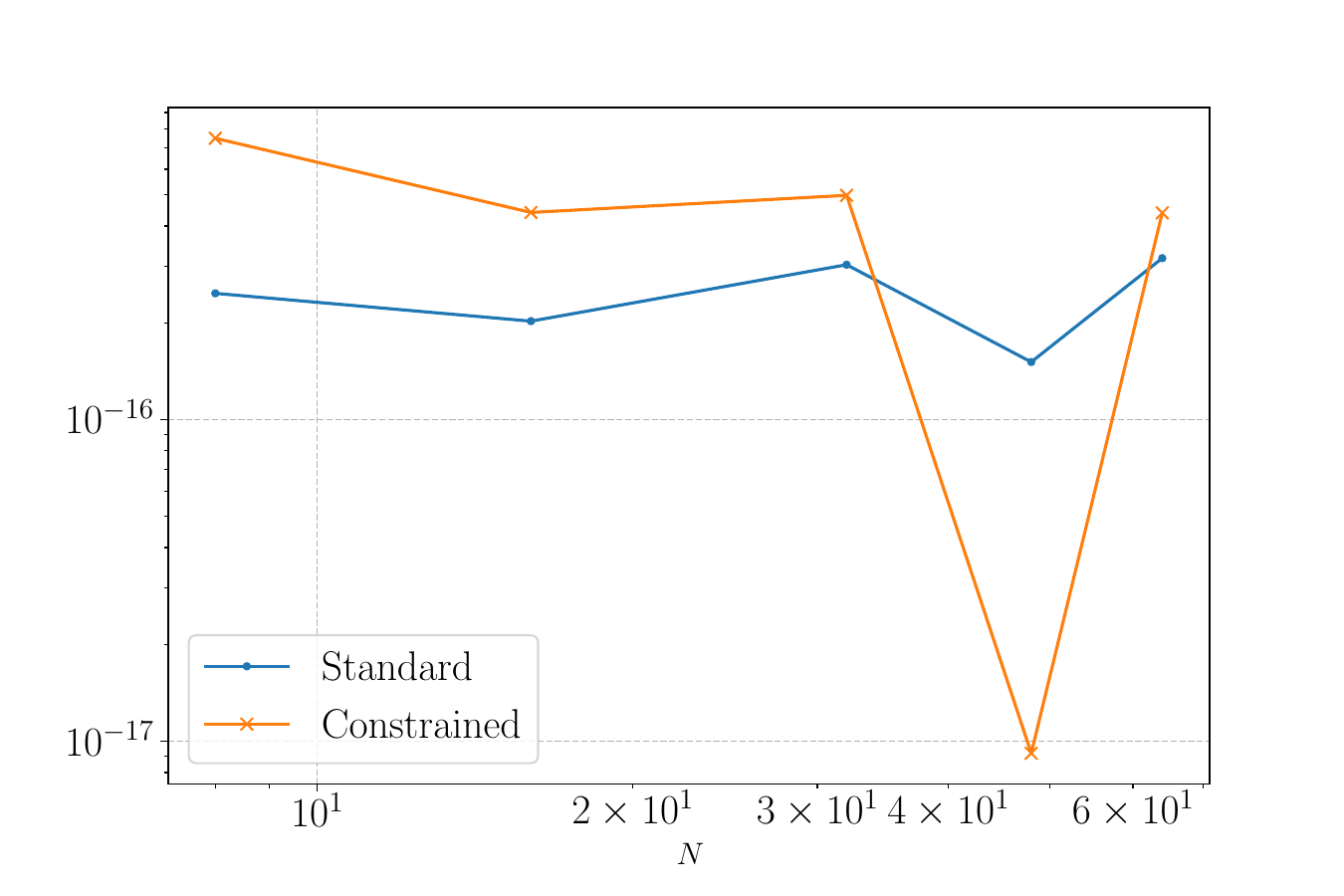}
   \includegraphics[width=.48\linewidth]{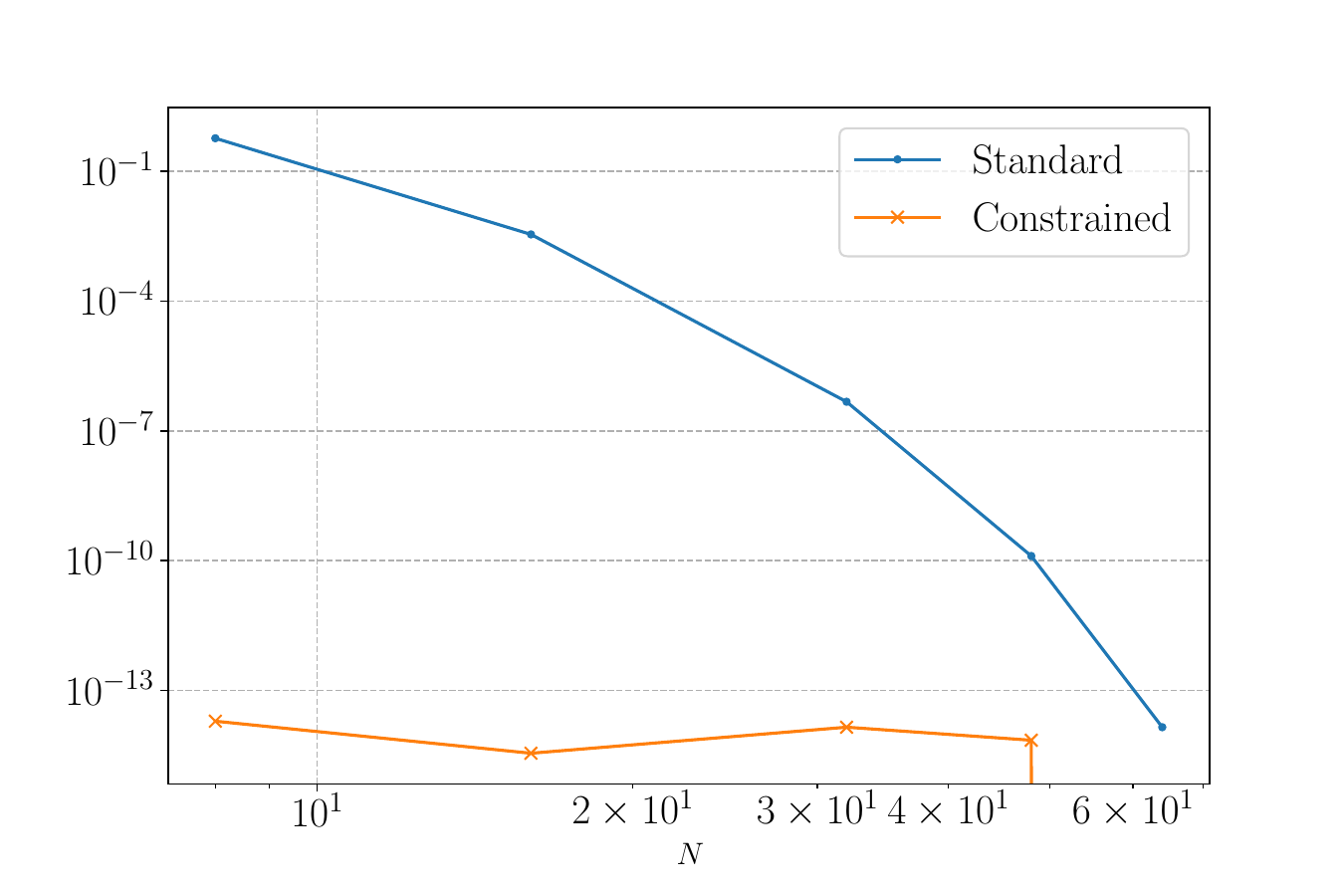}
  \label{ErrMomFokkerPlanck}
  \caption{Fokker-Planck model: Approximation error on the solution (Top Left) and absolute error on the mass (Top Right), mean velocity (Bottom Left) and temperature (Bottom Right) at $T_f=0.1$.}
\end{figure}

\paragraph{Long-time behaviour} Let us conclude this test by investigating the long time behaviour of the solution observed in Figure \ref{SnapFP}. In Figure \ref{LTB_FokkerPlanck} The top left panel shows the evolution of the norms $\|f_N(t^n)-f_\infty\|_2$ and $\|f_N^c(t^n)-f_\infty\|_2$ while the remaining panels show the variation of the moments $|\rho^n-\rho^0|$, $|\mu^n-\mu^0|$ and $|T^n-T^0|$. We observe that the constrained method indeed preserves the first three moments up to machine accuracy. As a consequence, the long-time behaviour should be better approximated. The standard approximation does seem to converge towards the steady state but ultimately, accumulation of errors on the moments overcomes the dynamic of the solution, pushing it away from the equilibrium. On the contrary, our proposed method does not exhibit this re-bounce and converges to a plateau that decreases as one increases the number of modes. One can mention that equilibrium preserving  techniques^^>\cite{FilbetPareschiRey2014, PareschiRey2020} have been developed to overcome this saturation but these are beyond the scope of this article.
\begin{figure}
  \centering
   \includegraphics[width=.48\linewidth]{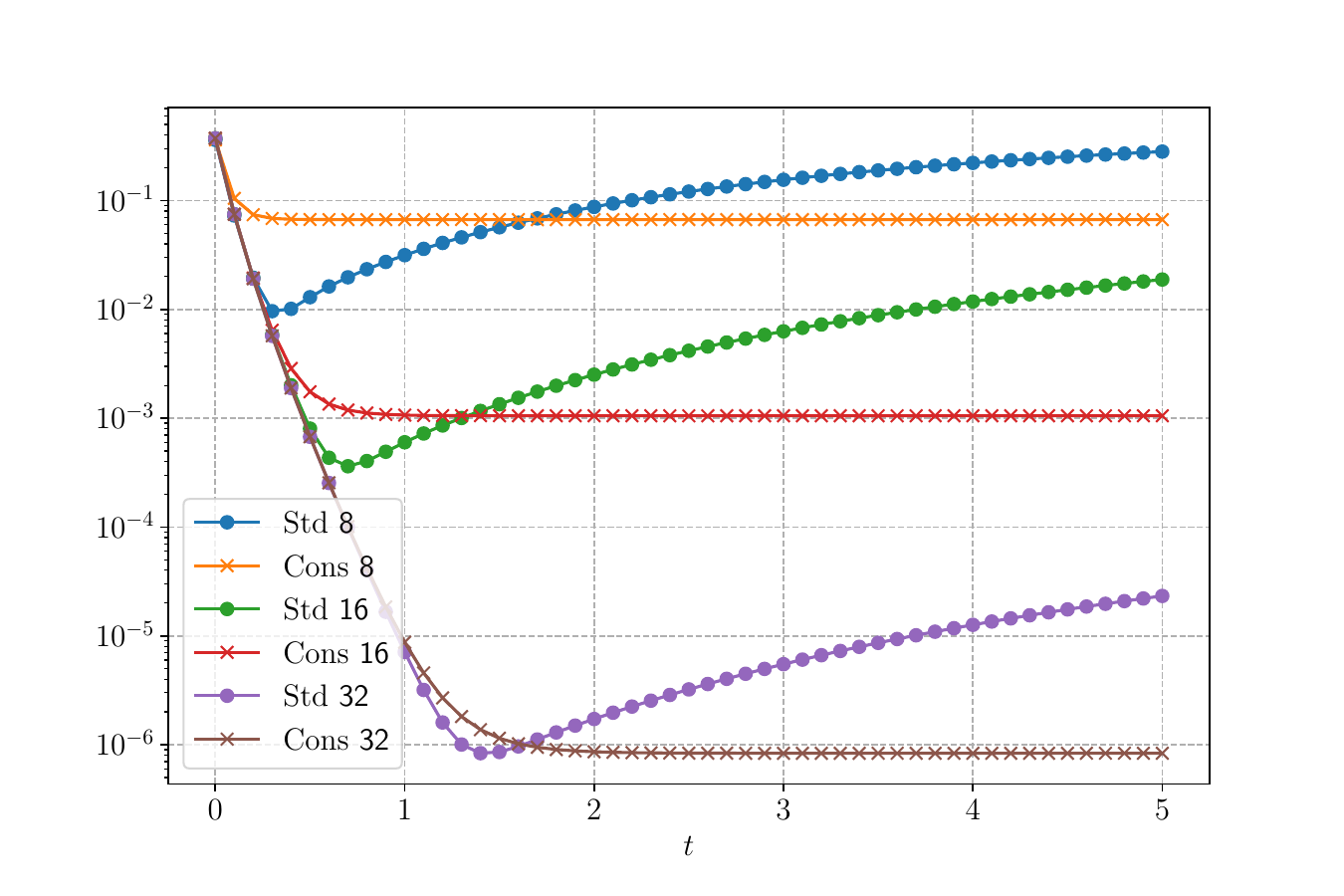}
   \includegraphics[width=.48\linewidth]{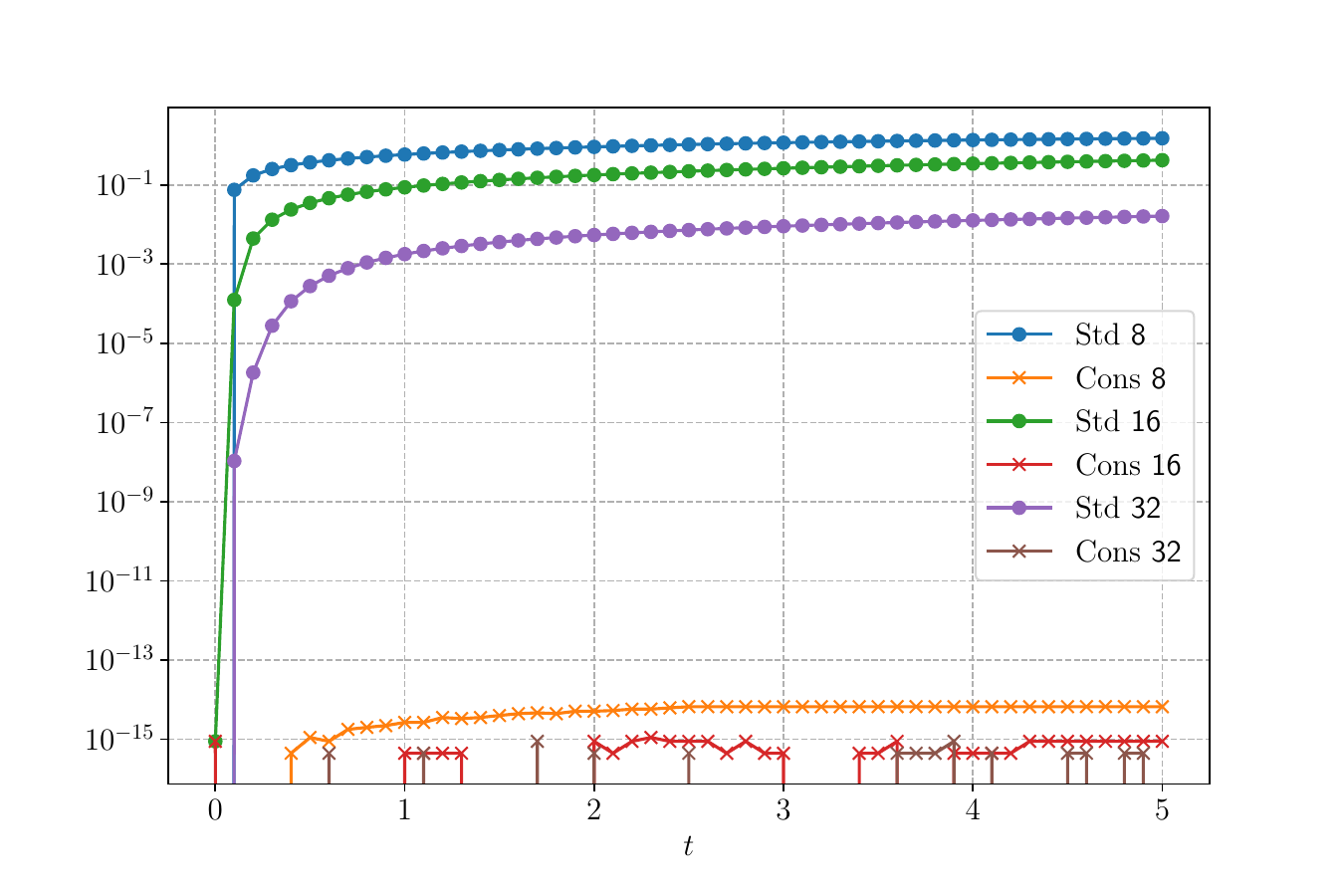}\\
   \includegraphics[width=.48\linewidth]{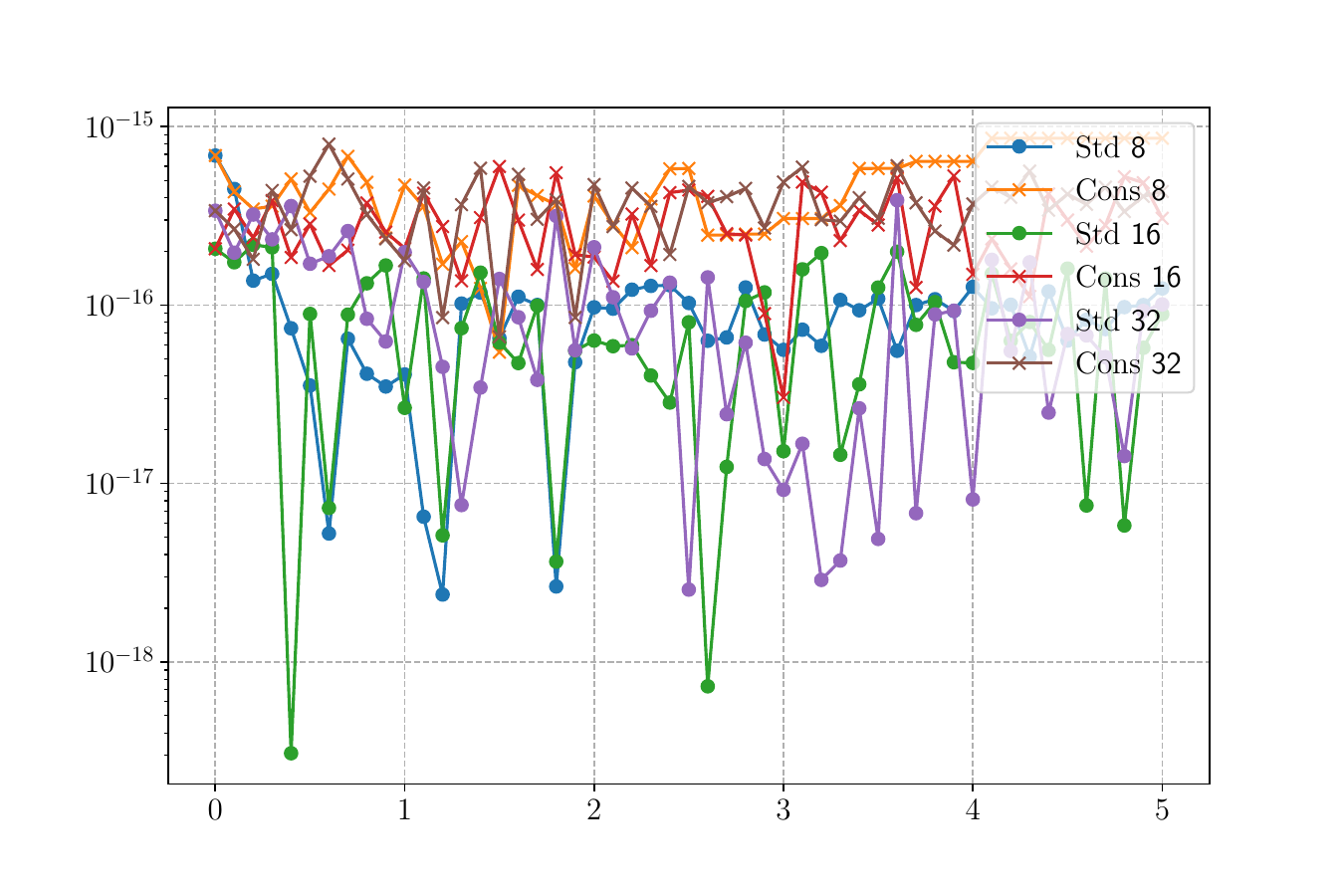}
   \includegraphics[width=.48\linewidth]{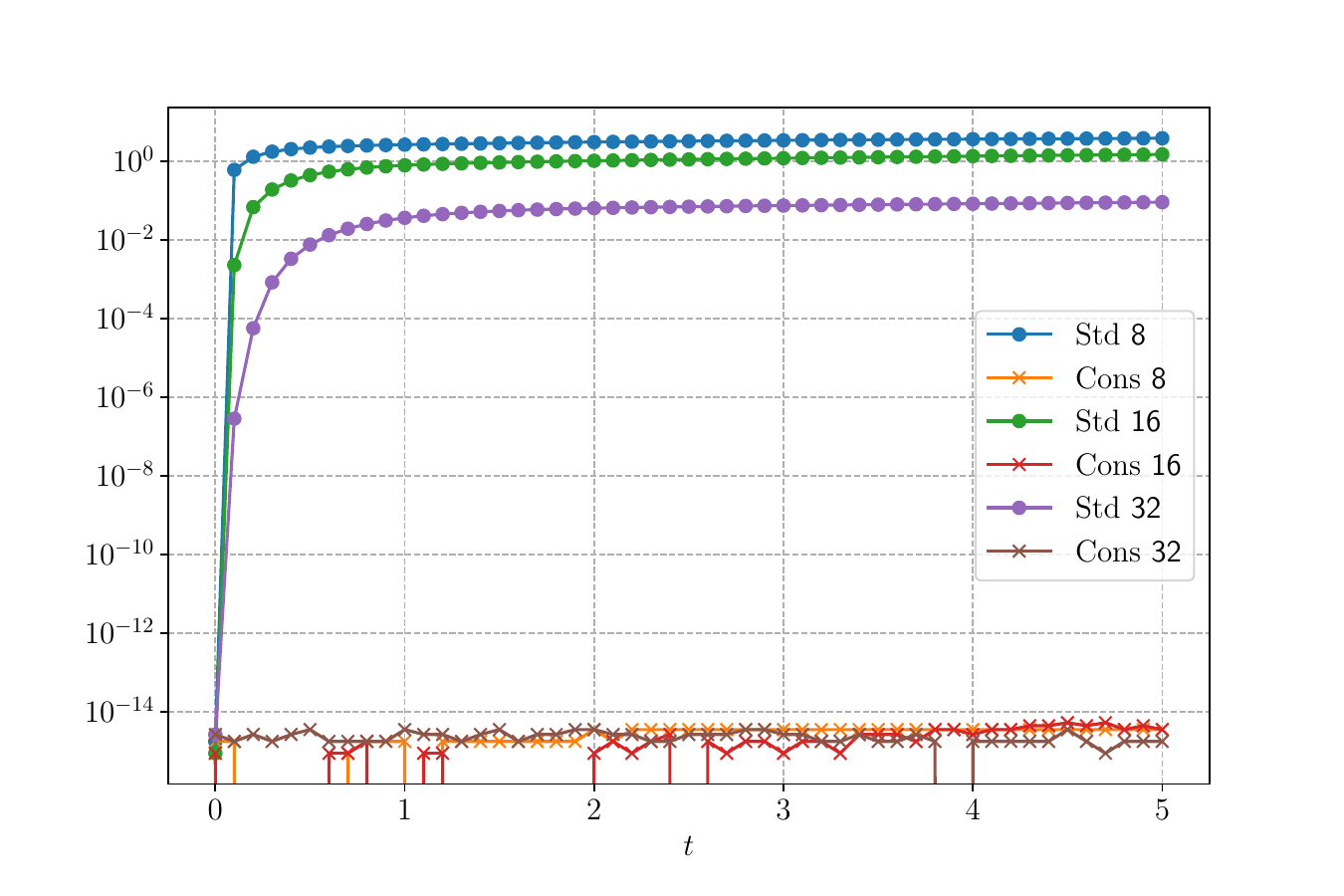}
  \label{LTB_FokkerPlanck}
  \caption{Fokker Planck model: Convergence towards the steady state (Top Left) and absolute error on the mass (Top Right), mean velocity (Bottom Left) and temperature (Bottom Right).}
\end{figure}


\subsection{Test 3: A model of Opinion formation}
For this second test,  we turn our attention to the following opinion model^^>\cite{FurioliPulvirentiTerraneoToscani2017} where $v\in[-1,1]$:
\be
\left\lbrace\begin{aligned}\label{Opinion}
  &\frac{\partial}{\partial t} g=\frac{\lambda}{2}\partial_{vv} \left((1-v^2)g\right)+\partial_v((v-m)g):=L_{Op}g,\\
  &g(0,v) = g^0(v).
\end{aligned}\right.
\ee
In this case, the unknown $g(t,v)$ describes the distribution of individuals at time $t$ with an opinion trait $v$. The modelling parameters are chosen so that $|m|<1$ and $\lambda<1+|m|$ to prevent blow-ups as $v\rightarrow\pm 1$. In addition, we supplement this equation with a no flux boundary condition to ensure the conservation of mass:
\be
  \frac{\lambda}{2}\partial_{v} \left((1-v^2)g\right)+(v-m)g\Big|_{v=\pm 1} = 0.
\ee
In particular, one can show that this condition in fact boils down to the homogeneous Dirichlet boundary condition $g(t,\pm 1)=0$. This equation also admits a steady state given by 
\be\label{OpinionEquilibrium}
  g_\infty(v) = c_{m,\lambda}(1+v)^{\frac{1+m}{\lambda}-1}(1-v)^{\frac{1-m}{\lambda}-1}
\ee
where the constant $c_{m,\lambda}$ is for normalization. In order to approximate solutions to \eqref{Opinion} we want to consider general Jacobi polynomials $\left(p^{\alpha,\beta}_k\right)_{k=0,\dots,N}$ associated to the weight $\omega(v)=(1-v)^\alpha(1+v)^\beta$. However, this polynomial basis does not satisfy the homogeneous Dirichlet boundary conditions. To solve this issue, we construct a new polynomial basis obtained from the original Jacobi polynomials (see \textit{e.g.}^^>\cite{ShenTangWang2011}). For $\alpha,\beta >-1$, let
\be
  \zeta_k = p^{\alpha,\beta}_k + a_kp^{\alpha,\beta}_{k+1} + b_kp^{\alpha,\beta}_{k+2},\quad k=0,\dots,N-2
\ee
where the constants $a_k$ and $b_k$ are solution to the linear system:
\be
\left\lbrace\begin{aligned}
  \zeta_k(-1) &= 0\\
  \zeta_k(1) &= 0.
\end{aligned}\right.
\ee
This new polynomial family now satisfies the boundary conditions but it is, by construction, no longer orthogonal. Since orthogonality is a key ingredient of the conservative spectral method, we need to apply a Gram-Schmidt algorithm to obtain an orthonormal basis. Note that in practice the process of orthonormalization may induce an accumulation of machine errors that can lead to a significant loss in orthogonality for large number of modes. We can now expand $g$ in the basis of polynomials $\left(\zeta_k\right)_{k=0,\dots,N-2}$:
\be
  g(t,v) = \sum_{k=0}^{N-2} \hat{g}_k(t)\zeta_k(v),
\ee
and then proceed as before to solve \eqref{Opinion} using a Galerkin method. We define the standard and constrained problems as
\begin{equation}\label{eq:galerkinOpinion}
    \left\lbrace\begin{aligned}
    &\frac{\partial}{\partial t}g_N = \mathcal{P}_N\left(L_{Op}(g_N)\right),\\
    &g_N(0,v)=\mathcal{P}_N(g^0)(v),
\end{aligned}\right.
\end{equation}
and
\begin{equation}\label{eq:galerkinOpinionc}
    \left\lbrace\begin{aligned}
    &\frac{\partial}{\partial t}g_N^c = L_{Op,N}^c(g_N^c),\\
    &g_N(0,v)=\mathcal{P}_N^c(g^0)(v).
\end{aligned}\right.
\end{equation}
Since \eqref{Opinion} preserves only the first moment, we define $L_{Op,N}^c$ as the solution to
\begin{equation}\label{eq:constrOperatorOpinion}
    L_{Op,N}^c(g) = {\rm argmin} \left\{\| g_N - L_{Op}(g) \|^2_{L^2_\omega}\, :\,g_N\in S_N,\,\, \langle g_N,1 \rangle=0\right\}.
\end{equation}
The time steeping is again achieved using an RK4 time integrator with time step $\Delta t =10^{-4}$. Before presenting numerical results, let us mention that the space spanned by the newly constructed basis does not contain the monomials $v^q$, $q=0,1,2,\dots$ anymore. As a consequence, in regards of \eqref{eq:MomProjError_weight}, even if one considers Legendre polynomials ($\omega=1$) to construct the new basis, one does not expect exact conservation of moments. However, as will be shown through numerical experiments, one still recovers spectral accuracy on the moments for the standard method which ensures spectral accuracy on the constrained solution. In regards of this discussion, we now present only the choice of Legendre polynomials for the construction of the new basis $\left(\zeta_k\right)_{k=0,\dots,N-2}$. Note however that classical families such as Chebyshev first and second kind provide similar results. As an initial data, we take 
\begin{equation}
    g^0(v) = c_0\left((1+v)^{12}(1-v)^6 + (1+v)^{13}(1-v)^{25}\right),
\end{equation}
where $c_0$ is for normalizing the first moment. Then we set $m=0$ and $\lambda=0.1$ in the definition of $L_{Op}$ in \eqref{Opinion}, therefore expecting an equilibrium of the form
\begin{equation*}
    g_\infty(v) = c_\infty(1-v^2)^9.
\end{equation*}
Figure \ref{SnapOpinion} shows the evolution of the approximate distributions $g_N$ and $g_N^c$. As in the previous case, we observe a very good agreement between the standard and constrained method with the two curves overlapping. Moreover, the solutions also appear to converge towards the steady state.
\begin{figure}
  \centering
   \includegraphics[width=.9\linewidth]{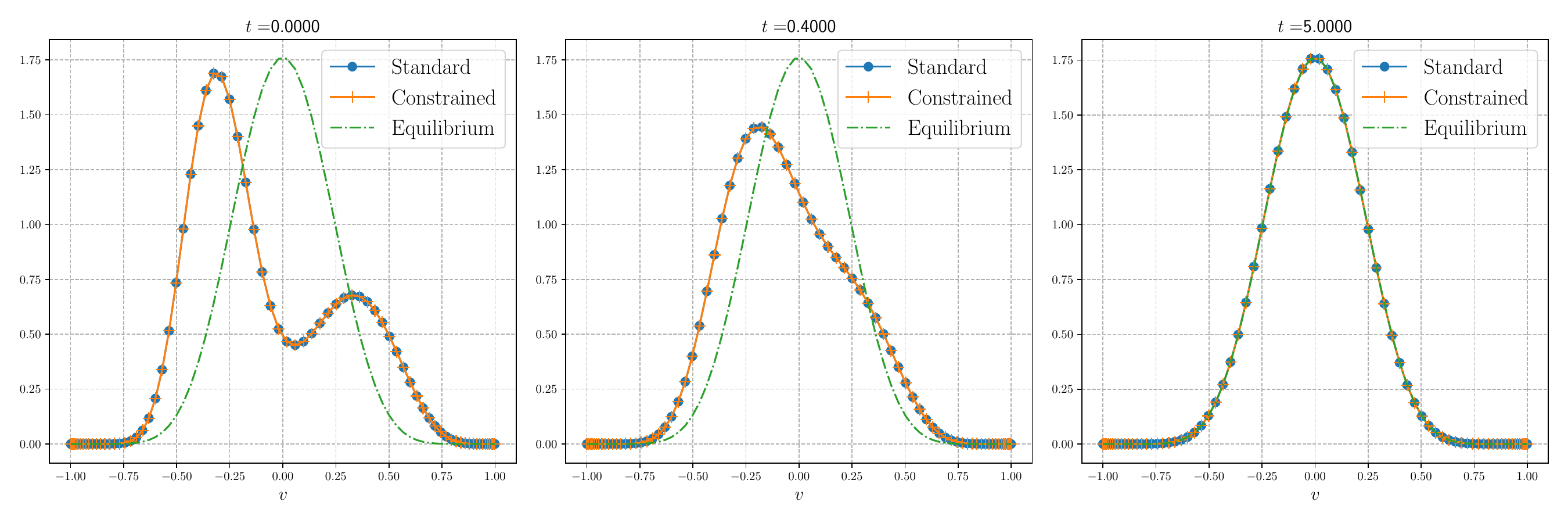}
  \label{SnapOpinion}
  \caption{Opinion model: Snapshots of the distribution at times $t=0,0.4$ and $5$, $N=24$.}
\end{figure}

\paragraph{Spectral accuracy and conservations} Let us now investigate the accuracy of the method and its mass conservation properties. Indeed, at the continuous level, equation \eqref{Opinion} along with no flux boundary conditions preserves only the $0th$ order moment. In the next study, we set a final time $T_f=0.1$ and compare the approximation error between the standard an constrained method. We observe in Figure \ref{ErrMomOpinion} that we again fall into the framework of the convergence theorem \ref{thm:SpectralAccuracyfNc} and the constrained method behaves as predicted.

\begin{figure}
  \centering
   \includegraphics[width=.48\linewidth]{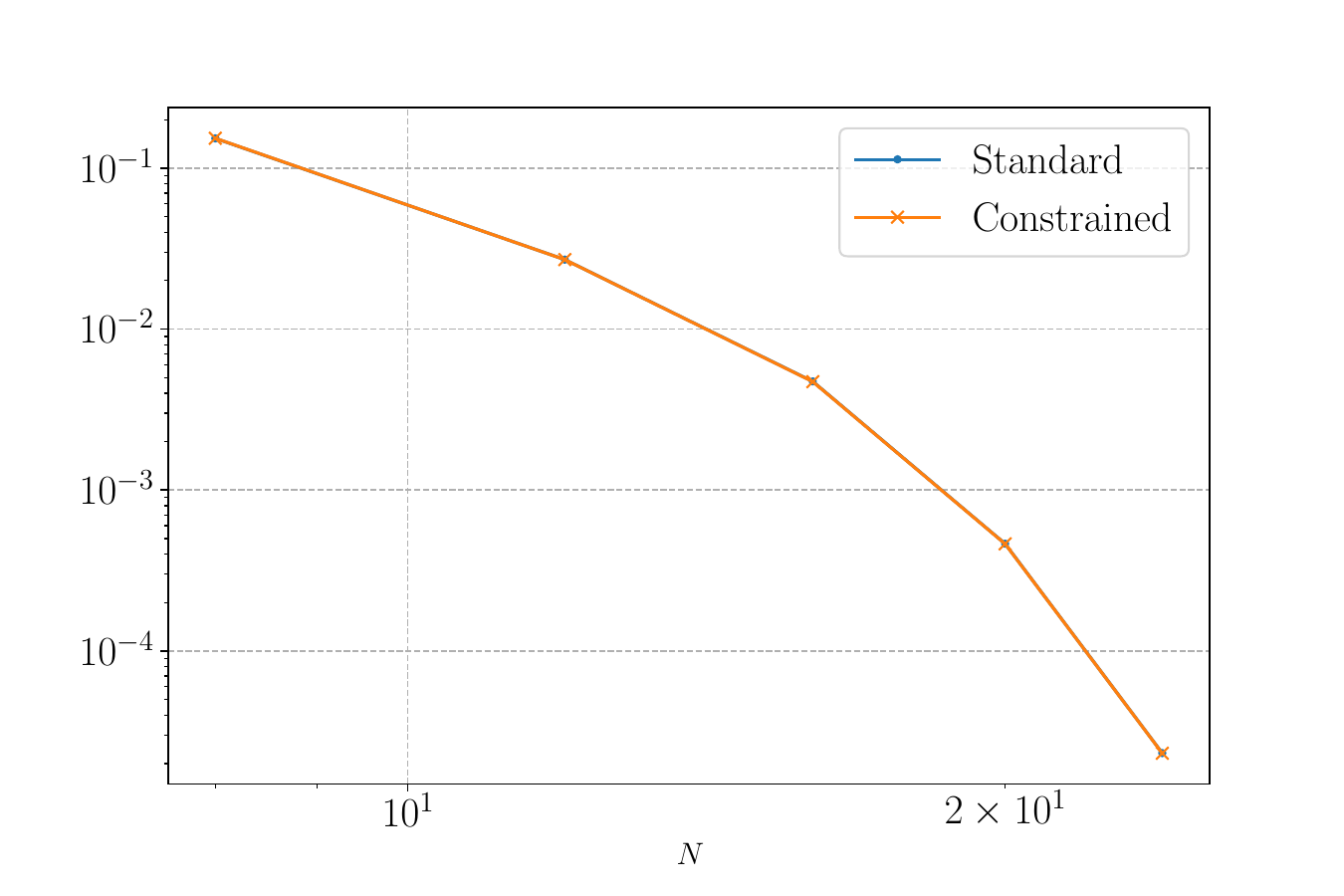}
   \includegraphics[width=.48\linewidth]{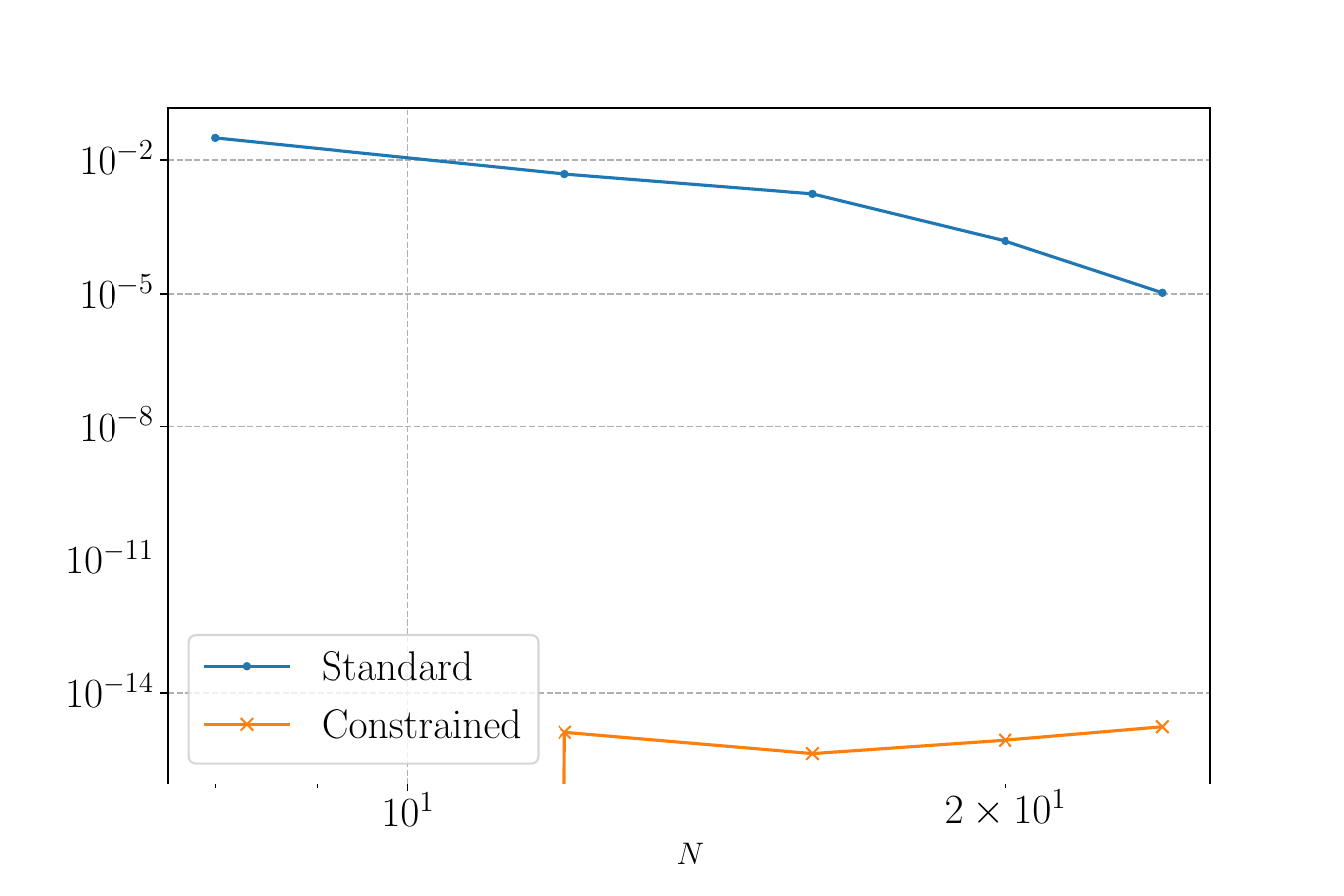}
  \label{ErrMomOpinion}
  \caption{Opinion model: Approximation error on the solution (Left) and on the 0\textsuperscript{th} order moment (Right) at $T_f=0.1$.}
\end{figure}

\paragraph{Long-time behaviour} Let us conclude this test by looking at the long time behaviour of the approximated solution to \eqref{Opinion}. In order to quantify the observation made on Figure \ref{SnapOpinion}, we show in Figure \ref{LTB_Opinion} the evolution of the norms $\|g_N-g_\infty\|_2$ and $\|g_N^c-g_\infty\|_2$ as well as the evolution of the mass variation $|m_0^0-m_0(t^n)|$. Similarly as before, we observe that the conservation of moments allows to stabilize the long time dynamic of the solution.

\begin{figure}
  \centering
   \includegraphics[width=.48\linewidth]{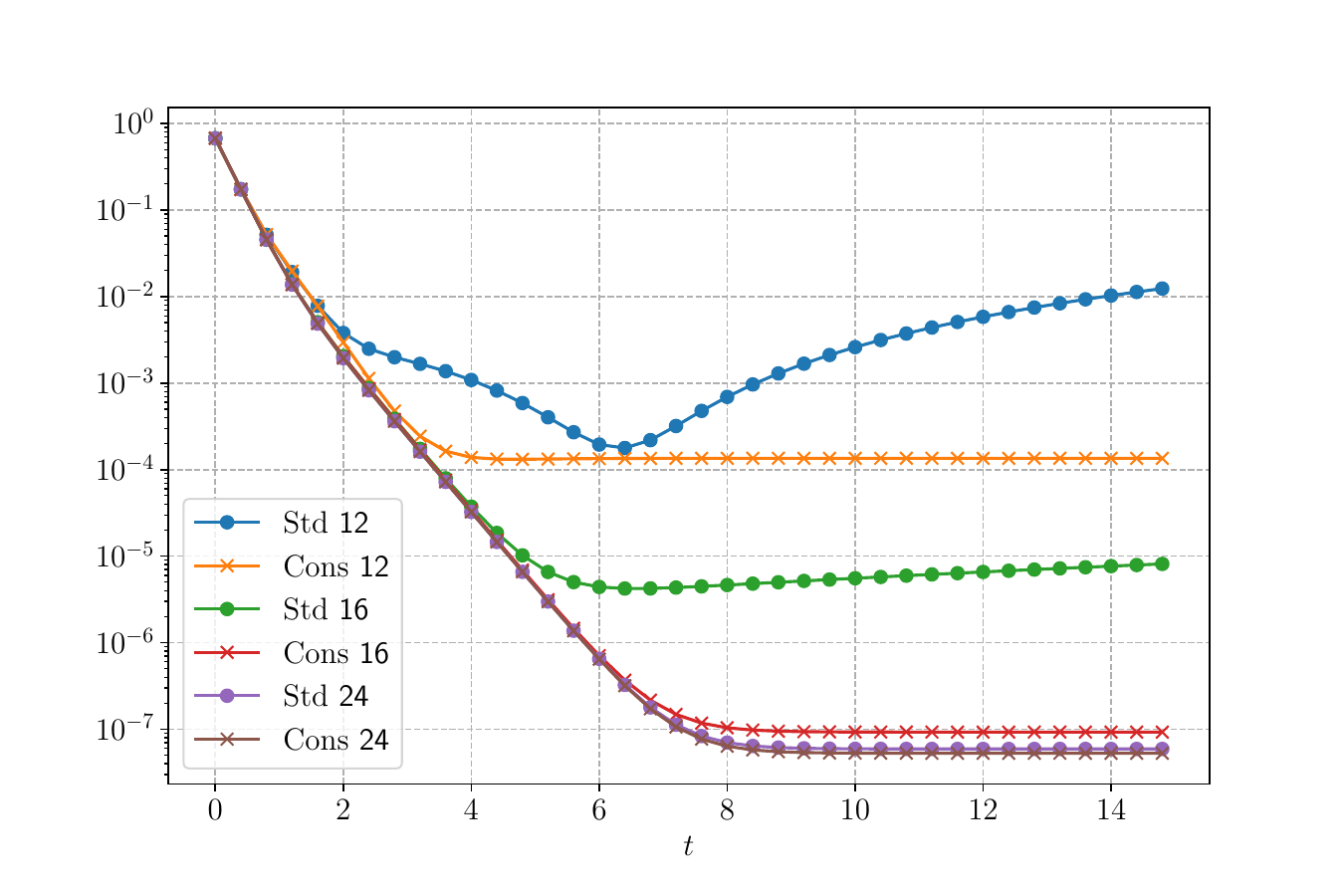}
   \includegraphics[width=.48\linewidth]{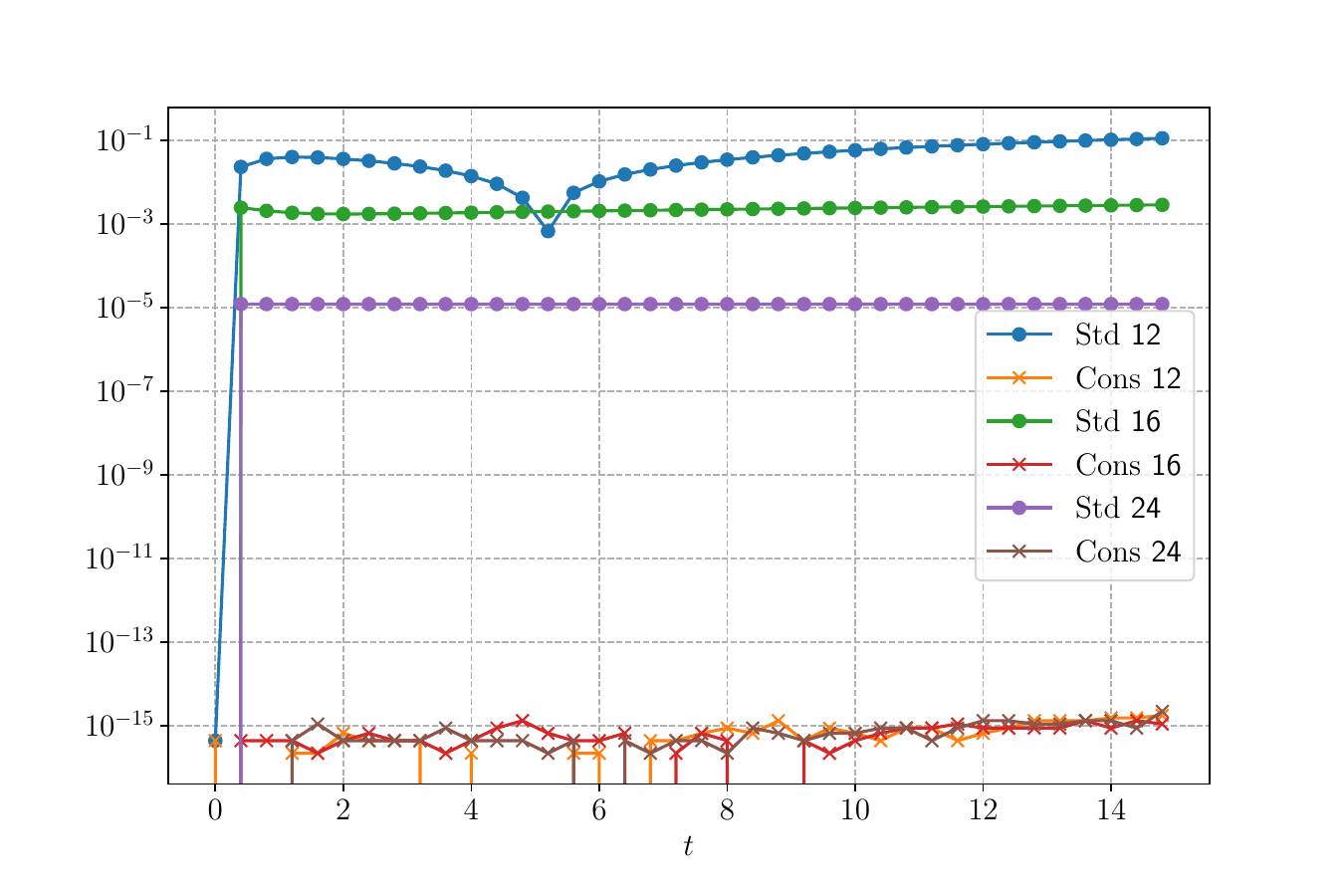}
  \label{LTB_Opinion}
  \caption{Opinion model: Convergence towards the steady state (Left) and mass variation (Right).}
\end{figure}


\subsection{Test 4: A Call center service time model}
This final test deals with a call center service time model introduced in^^>\cite{GualandiToscani2018}:
\be\label{CallCenter}
\left\lbrace\begin{aligned}
  &\frac{\partial}{\partial t} h=\frac{\lambda}{2}\partial_{vv} \left(v^2h\right)+\frac{\gamma}{2}\partial_v\left(v\ln\left(\frac{v}{v_L}\right)h\right) := L_{CC} h,\\
  &h(0,v)=h^0(v).
\end{aligned}\right.
\ee
In this context, the distribution density $h(t,v)$ describes the service time of agents in a call center. The constants $0<\gamma<1$, $\lambda>0$ and $v_L>0$ are modelling parameters. In particular, the latter corresponds to an ideal time for agents to complete their task. To supplement this equation, we consider a no-flux boundary condition
\be
    \frac{\lambda}{2}\partial_{v} \left(v^2h\right)+\frac{\gamma}{2}\ln\left(\frac{v}{v_L}\right)h\Big|_{v=0}=0.
\ee
It is clear that this condition is automatically satisfied by \eqref{CallCenter}.
It can also be shown that \eqref{CallCenter} admits an equilibrium distribution of the form:
\be\label{CallCenterEquilibrium}
  h_\infty(v) = \frac{1}{\sqrt{2\pi\sigma}v}\exp\left(-\frac{(\ln(v)-\mu)^2}{2\sigma}\right)
\ee
where $\sigma=\frac{\lambda}{\gamma}$ and $\mu=\ln(v_L)-\sigma$. As in Test 1, in order to deal with integrability at infinity, we want to consider the symmetrically weighted Laguerre functions: $\xi_k=\mathcal{L}_k\omega^\frac{1}{2}$. In addition, to deal with the poor approximation of the steady state observed in Figure \ref{fig:ConvLaguerreBad} we consider a micro-macro approach. Since the equilibrium is known, we can decompose the unknown as
\begin{equation}\label{CallCenterMiMa}
    h = h_\infty + \Tilde{h},
\end{equation}
where $\Tilde{h}$ is the perturbation that can take negative values. Since \eqref{CallCenter} is linear in $h$, and by definition of the steady state, the perturbation also solves \eqref{CallCenter}. The only difference is the initial data.
\begin{equation*}
\left\lbrace\begin{aligned}
  &\frac{\partial}{\partial t} \Tilde{h}=\frac{\lambda}{2}\partial_{vv} \left(v^2\Tilde{h}\right)+\frac{\gamma}{2}\partial_v\left(v\ln\left(\frac{v}{v_L}\right)\Tilde{h}\right),\\
  &\Tilde{h}(0,v)=h^0(v)-h_\infty(v).
\end{aligned}\right.
\end{equation*}
From this procedure, we expect to observe that $\Tilde{h}$ converges towards $0$ for long time. Consequently, the long time behaviour of the solution should be better approximated since $0$ now belongs to the approximation space. However, one cannot say anything about the short time behaviour as the initial data could still be poorly approximated.

As previously, we can then define the standard and constrained problems on the perturbation as
\begin{equation}\label{eq:galerkinCallCenter}
    \left\lbrace\begin{aligned}
    &\frac{\partial}{\partial t}\Tilde{h}_N = \mathcal{P}_N\left(L_{CC}(\Tilde{h}_N)\right),\\
    &\Tilde{h}_N(0,v)=\mathcal{P}_N(\Tilde{h}^0-h_\infty)(v),
\end{aligned}\right.
\end{equation}
and
\begin{equation}\label{eq:galerkinCallCenterc}
    \left\lbrace\begin{aligned}
    &\frac{\partial}{\partial t}\Tilde{h}_N^c = L_{CC,N}^c(\Tilde{h}_N^c),\\
    &\Tilde{h}_N(0,v)=\mathcal{P}_N^c(\Tilde{h}^0-h_\infty)(v).
\end{aligned}\right.
\end{equation}
The original distribution is then reconstructed using \eqref{CallCenterMiMa} for visualisations. Since \eqref{CallCenter} preserves only the first moment, we define $L_{CC,N}^c$ as the solution to
\begin{equation}\label{eq:constrOperatorCallCenter}
    L_{CC,N}^c(h) = {\rm argmin} \left\{\| g_N - L_{CC}(h) \|^2_{L^2_\omega}\, :\,g_N\in S_N,\,\, \langle g_N,1 \rangle=0\right\}.
\end{equation}
The time stepping is yet again achieved through a RK4 method with time step $\Delta t=10^{-4}$. Let us now fix the parameters $\lambda=0.5$, $\gamma=0.9$ and $v_L=40$ in \eqref{CallCenter}. As an initial data, we consider
\be
    h^0(v) = (x^3-2x+\sin(x))e^{-x}.
\ee
This function is actually well approximated by Laguerre functions. We show on Figure \ref{SnapCallCenter} the evolution of the standard and constrained distributions. A first observation is that there is again a good agreement between the two mathod. In addition, as expected from the micro-macro approach, the steady state is well approximated.
\begin{figure}
  \centering
  \includegraphics[width=.9\linewidth]{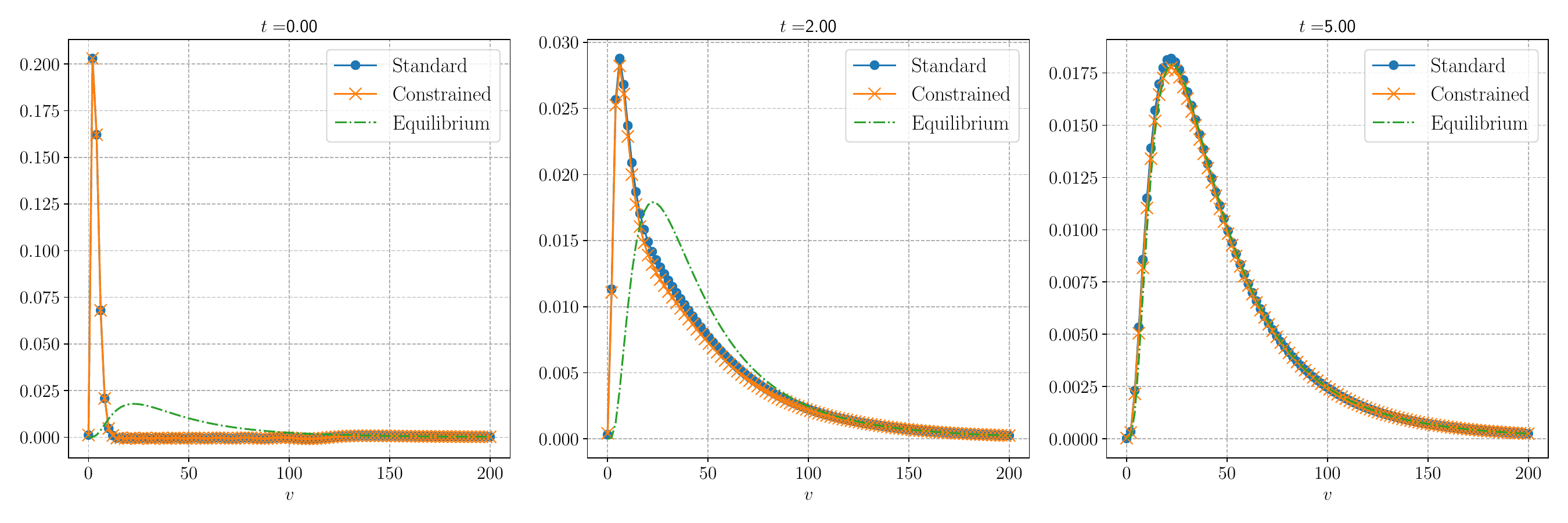}
  \label{SnapCallCenter}
  \caption{Call Center model: Snapshots of the distributions at times $t=0,2$ and $5$, $N=32$.}
\end{figure}

\paragraph{Accuracy and conservation}
To assess the spectral accuracy of the method, we consider a final time $T_f=0.1$ along with a time step $\Delta t=10^{-4}$ for the time integration. We present on Figure \ref{ErrMomCallCenter} the study of the convergence of the method. As expected, in short times, the distribution is not very well approximated in our chosen basis. As a consequence, we are only able to recover a convergence that is quite slow. However, the up-side is that the mass is still exactly preserved by the constrained approximation. Now, contrary to previous tests, it is important to note that the standard method is not spectrally accurate on the mass.

\begin{figure}
  \centering
  \includegraphics[width=.48\linewidth]{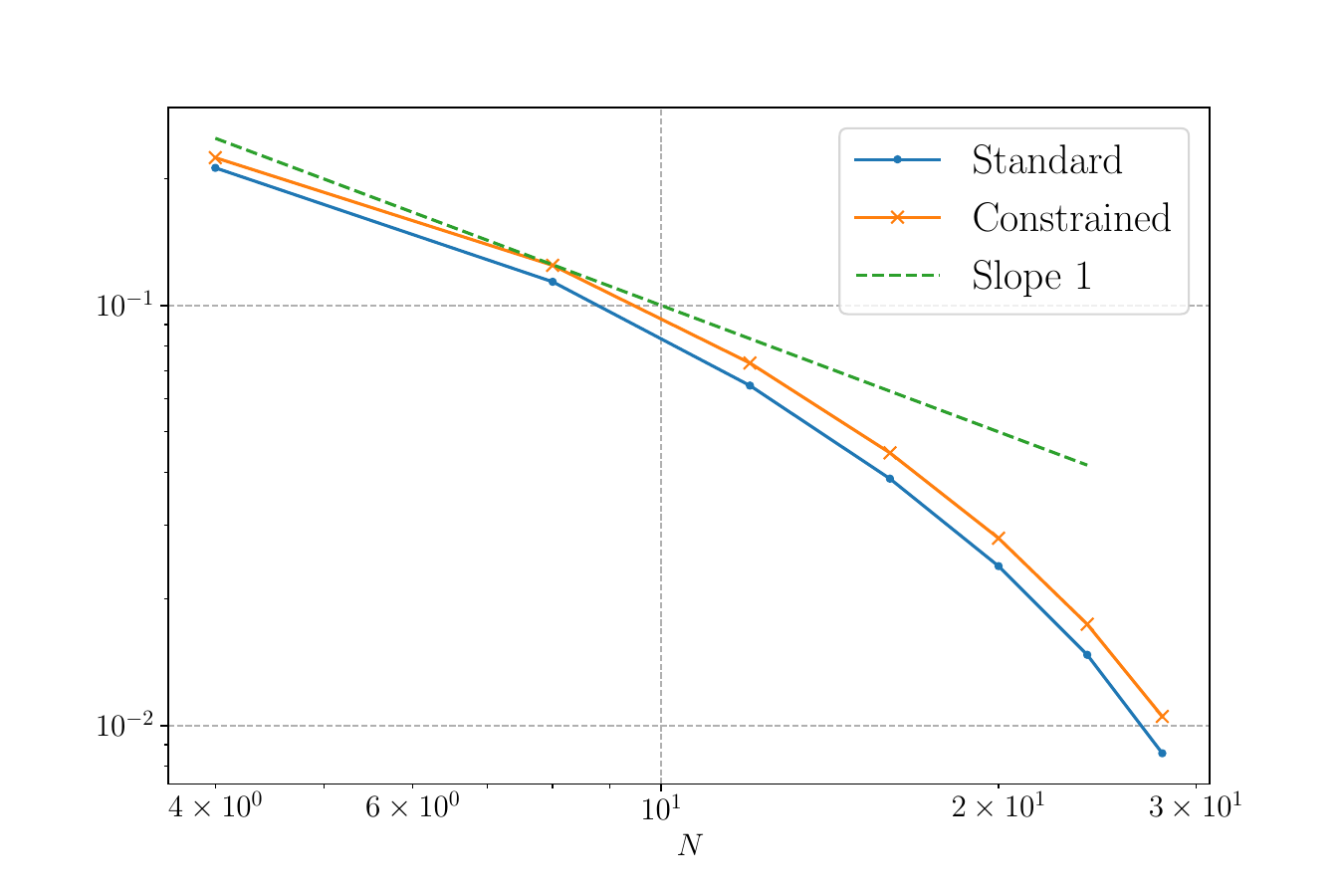}
  \includegraphics[width=.48\linewidth]{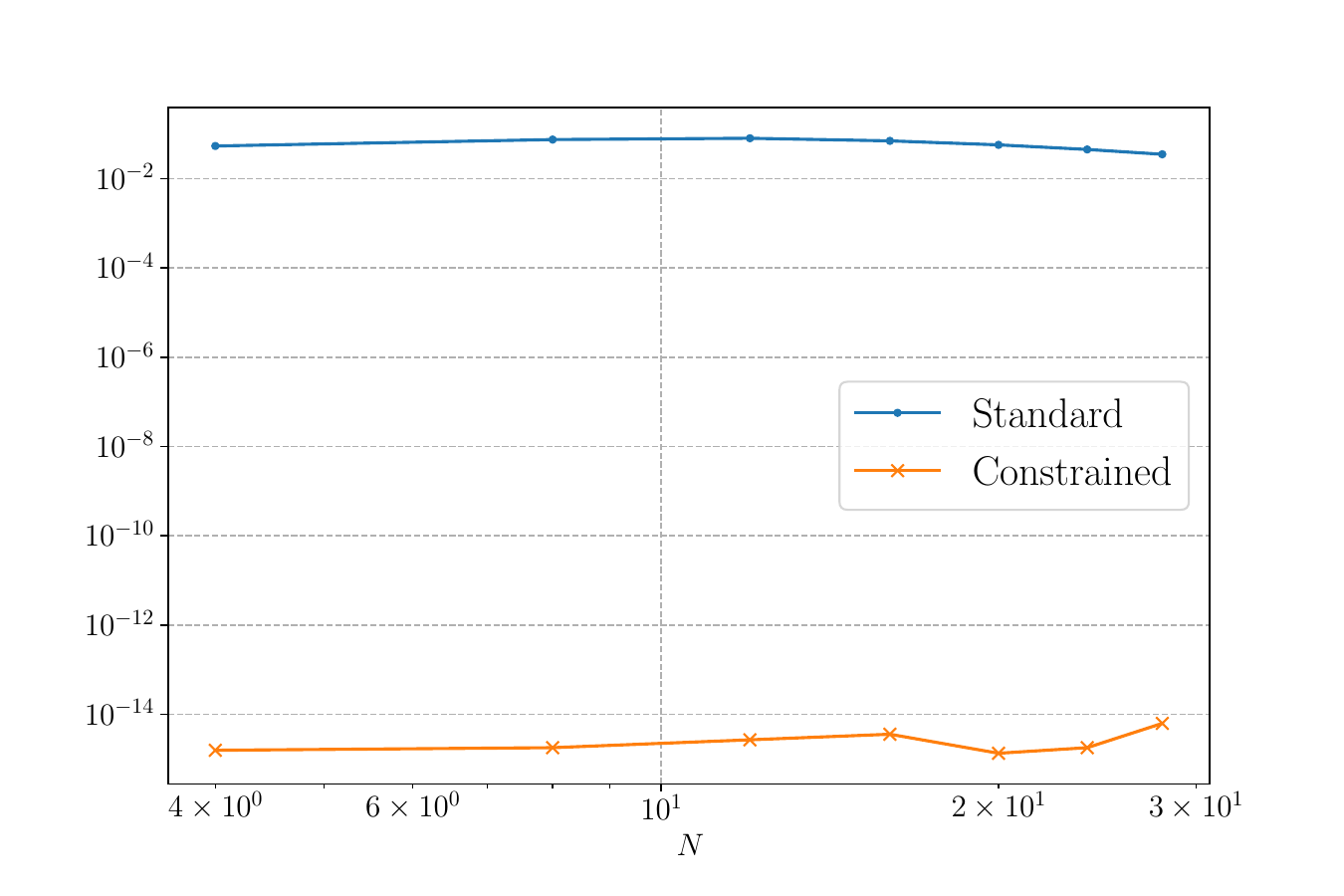}
  \label{ErrMomCallCenter}
  \caption{Call Center model: Approximation error on the solution (Left) and on the 0\textsuperscript{th} order moment (Right) at $T_f=0.1$.}
\end{figure}

\paragraph{Long-time behaviour} To conclude this section, and as in previous tests, we investigate the longtime behaviour of the solution through the norms $\|h_n-h_\infty\|_2$ and $\|h_n^c-h_\infty\|_2$. We observe in Figure \ref{LTB_CallCenter} that the micro-macro approach allows to approximate really well the steady state with both methods. In addition, the mass variation induced by the standard approach decays exponentially fast towards $0$ therefore ensuring that it also converges towards $h_\infty$ albeit more slowly than the conservative approach.

\begin{figure}
  \centering
  \includegraphics[width=.48\linewidth]{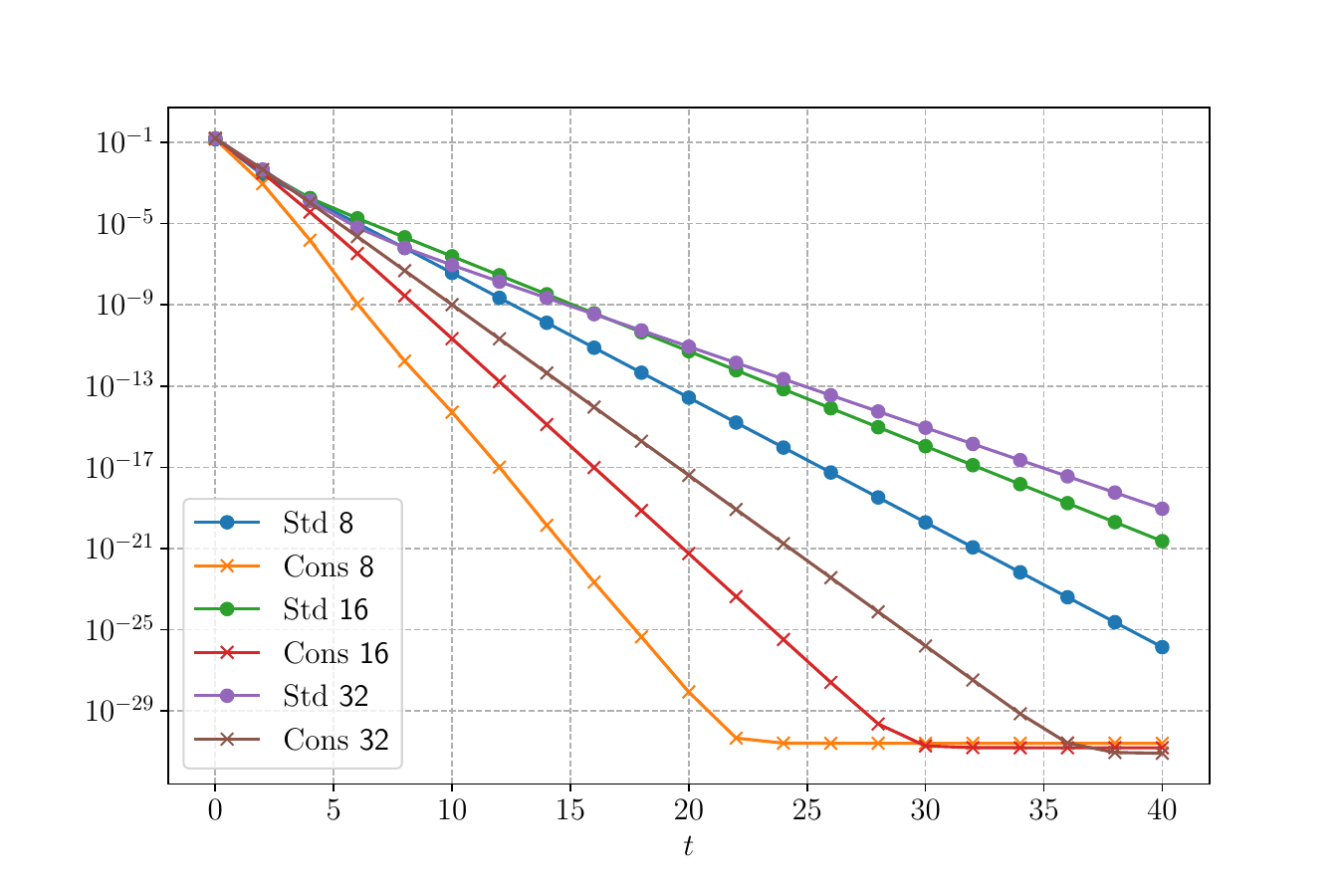}
  \includegraphics[width=.48\linewidth]{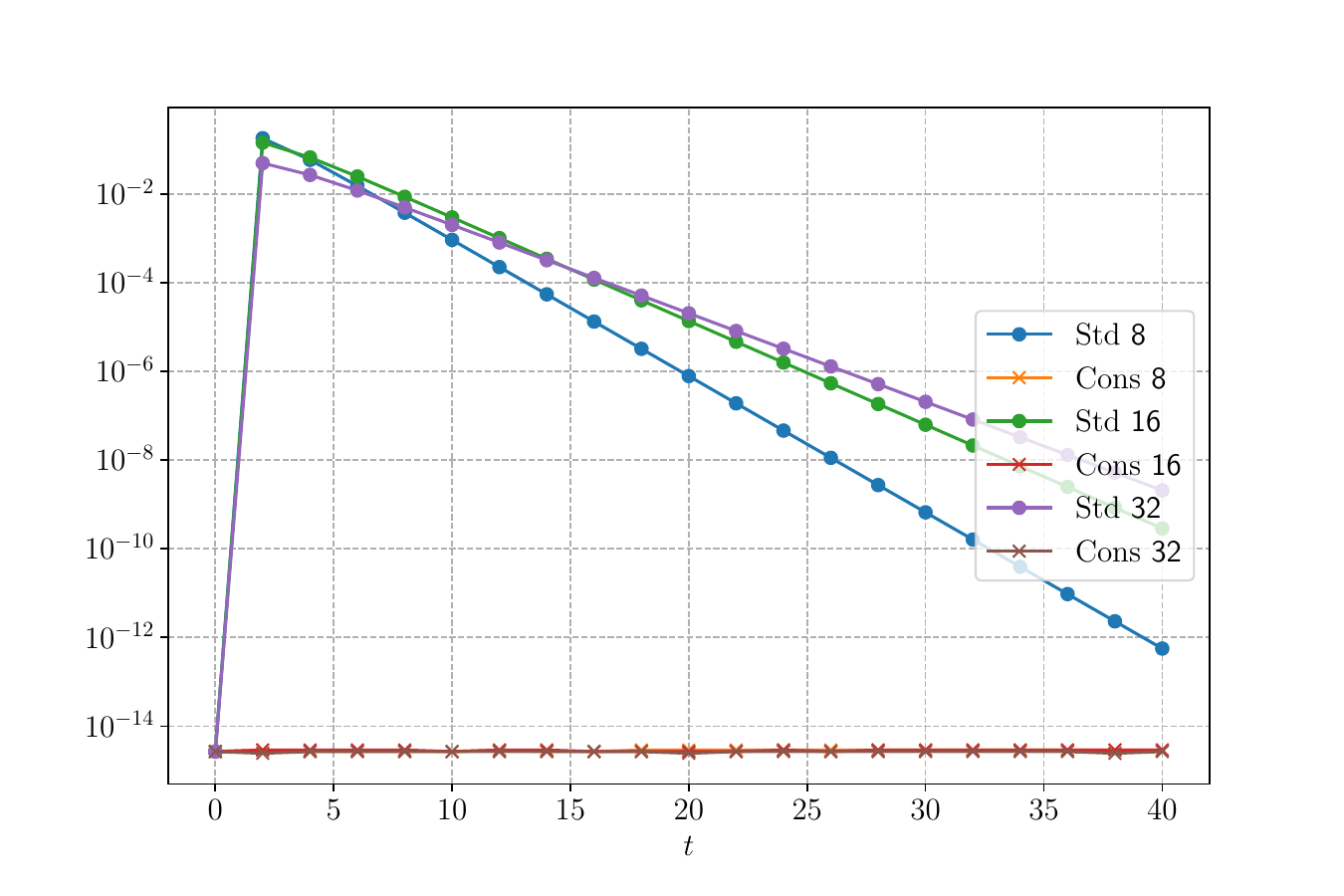}
  \label{LTB_CallCenter}
  \caption{Call Center model: Convergence towards the steady state (Left), and mass variation (Right).}
\end{figure}

\section{Conclusions}
We extended and analyzed a class of Galerkin spectral methods for PDEs based on general families of orthogonal functions that can preserve the moments of the solution. Leading examples of such problems arise in kinetic and mean-field theory where preservation of the moments of the distribution function is of paramount importance to describe correctly the long-time behavior. The methods were designed using an $L_2$ best constrained approximation formalism in the space of orthogonal polynomials. Thanks to this setting, the resulting methods allow the spectral accuracy property of the solution to be maintained for the conservative polynomial series. The new approximation was then used to derive moment preserving methods for several Fokker-Planck equations in distinct physical domains and by adopting different standard Galerkin projections, including the use of micro-macro decompositions.

Given its generality, the method admits numerous extensions. Among the most interesting are certainly the construction of spectrally accurate and conservative methods for the Vlasov equation. Another interesting direction is the application to kinetic models in the socio-economic domain where equilibrium states are often unknown and thanks to the present approach can be computed with spectral accuracy.

\medskip
\paragraph{Acknowledgment}
This work has been written within the activities of GNCS and GNFM groups of INdAM (Italian National Institute of High Mathematics). LP has been supported by the Royal Society under the Wolfson Fellowship ``Uncertainty quantification, data-driven simulations and learning of multiscale complex systems governed by PDEs". The partial support by ICSC -- Centro Nazionale di Ricerca in High Performance Computing, Big Data and Quantum Computing, funded by European Union -- NextGenerationEU and by MUR PRIN 2022 Project No. 2022KKJP4X ``Advanced numerical methods for time dependent parametric partial differential equations with applications" is also acknowledged. T.L. would like to thanks Thomas Rey for fruitful discussions on the implementation of the method. T.L. acknowledges support from Labex CEMPI (ANR-11-LABX-0007-01) and the MSCA DN-2022 program DATAHYKING.


\bibliographystyle{acm}
\bibliography{Orthpolycons}

\begin{thebibliography}{10}

\bibitem{Alonso2018}
{\sc Alonso, R.~J., Gamba, I.~M., and Tharkabhushanam, S.~H.}
\newblock {Convergence and error estimates for the Lagrangian-based conservative spectral method for Boltzmann equations}.
\newblock {\em SIAM J. Numer. Anal. 56}, 6 (2018), 3534--3579.

\bibitem{Bailo2020}
{\sc Bailo, R., Carrillo, J.~A., and Hu, J.}
\newblock Fully discrete positivity-preserving and energy-dissipating schemes for aggregation-diffusion equations with a gradient-flow structure.
\newblock {\em Commun. Math. Sci. 18}, 5 (2020), 1259--1303.

\bibitem{BessemoulinFilbet2021}
{\sc Bessemoulin-Chatard, M., and Filbet, F.}
\newblock {On the stability of conservative discontinuous Galerkin/Hermite spectral methods for the Vlasov-Poisson system}.
\newblock {\em ArXiv abs/2106.07468\/} (2021).

\bibitem{Buet2010}
{\sc Buet, C., and Dellacherie, S.}
\newblock On the {C}hang and {C}ooper scheme applied to a linear {F}okker-{P}lanck equation.
\newblock {\em Commun. Math. Sci. 8}, 4 (2010), 1079--1090.

\bibitem{canuto1988}
{\sc Canuto, C., Hussaini, M., Quarteroni, A., and Zang, T.}
\newblock {\em {Spectral Methods in Fluid Dynamics}}.
\newblock Springer Series in Computational Physics. Springer-Verlag, New York, 1988.

\bibitem{Chang1970}
{\sc Chang, J., and Cooper, G.}
\newblock A practical difference scheme for {Fokker-Planck} equations.
\newblock {\em Journal of Computational Physics 6}, 1 (1970), 1--16.

\bibitem{DimarcoPareschi2015}
{\sc Dimarco, G., Li, Q., Pareschi, L., and Yan, B.}
\newblock Numerical methods for plasma physics in collisional regimes.
\newblock {\em J. Plasma Phys. 81}, 1 (2015), 305810106.

\bibitem{FilbetPareschiRey2014}
{\sc Filbet, F., Pareschi, L., and Rey, T.}
\newblock {On steady-state preserving spectral methods for homogeneous Boltzmann equations}.
\newblock {\em Comptes Rendus Mathematique 353\/} (2014), 309--314.

\bibitem{FokGuoTang2002}
{\sc Fok, J. C.~M., yu~Guo, B., and Tang, T.}
\newblock {Combined Hermite spectral-finite difference method for the Fokker-Planck equation}.
\newblock {\em Math. Comput. 71\/} (2002), 1497--1528.

\bibitem{funaro1992}
{\sc Funaro, D.}
\newblock {\em {Polynomial approximations of differential equations}}.
\newblock Lecture Notes in Physics. Springer-Verlag, New York, 1992.

\bibitem{FurioliPulvirentiTerraneoToscani2017}
{\sc Furioli, G., Pulvirenti, A., Terraneo, E., and Toscani, G.}
\newblock Fokker–planck equations in the modeling of socio-economic phenomena.
\newblock {\em Mathematical Models and Methods in Applied Sciences 27\/} (2017), 115--158.

\bibitem{Gamba2017}
{\sc Gamba, I.~M., Haack, J.~R., Hauck, C.~D., and Hu, J.}
\newblock {A fast spectral method for the Boltzmann collision operator with general collision kernels}.
\newblock {\em SIAM J. Sci. Comput. 39}, 4 (2017), B658--B674.

\bibitem{Gamba2010}
{\sc Gamba, I.~M., and Tharkabhushanam, S.~H.}
\newblock {Shock and Boundary Structure formation by Spectral-Lagrangian methods for the Inhomogeneous Boltzmann Transport Equation}.
\newblock {\em J. Comput. Math. 28}, 4 (2010), 430--460.

\bibitem{Gautschi2004}
{\sc Gautschi, W.}
\newblock {\em {Orthogonal Polynomials: Computation and Approximation}}.
\newblock Oxford University Press, Oxford, 2004.

\bibitem{Gosse2013}
{\sc Gosse, L.}
\newblock {\em Computing Qualitatively Correct Approximations of Balance Laws. {E}xponential-Fit, Well-Balanced and Asymptotic-Preserving}.
\newblock SEMA SIMAI Springer Series. Springer, 2013.

\bibitem{GottliebOrszag1977}
{\sc Gottlieb, D., and Orszag, S.~A.}
\newblock {\em Numerical Analysis of Spectral Methods: Theory and Applications}.
\newblock SIAM, Philadelphia, 1977.

\bibitem{GualandiToscani2018}
{\sc Gualandi, S., and Toscani, G.}
\newblock {Call center service times are lognormal: A Fokker–Planck description}.
\newblock {\em Mathematical Models and Methods in Applied Sciences\/} (2018).

\bibitem{HWL2006}
{\sc Hairer, E., Wanner, G., and Lubich, C.}
\newblock {\em Geometric Numerical Integration: Structure-Preserving Algorithms for Ordinary Differential Equations}, vol.~31 of {\em Springer Series in Computational Mathematics}.
\newblock Springer, 2006.

\bibitem{Holloway1996}
{\sc Holloway, J.~P.}
\newblock {Spectral velocity discretizations for the Vlasov-Maxwell equations}.
\newblock {\em Transport Theory and Statistical Physics 25\/} (1996), 1--32.

\bibitem{Hu2019}
{\sc Hu, J., and Ma, Z.}
\newblock A fast spectral method for the inelastic {B}oltzmann collision operator and application to heated granular gases.
\newblock {\em J. Comput. Phys. 385\/} (2019), 119--134.

\bibitem{Jin2022}
{\sc Jin, S.}
\newblock {Asymptotic-preserving schemes for multiscale physical problems}.
\newblock {\em Acta Numerica 31\/} (2022), 415--489.

\bibitem{KoekoekLeskySwarttouw2010}
{\sc Koekoek, R., Lesky, P.~A., and Swarttouw, R.~F.}
\newblock {\em {Hypergeometric Orthogonal Polynomials and their q-Analogues}}.
\newblock Springer, Berlin, 2010.

\bibitem{Kraus2017}
{\sc Kraus, M., and Hirvijoki, E.}
\newblock {Metriplectic integrators for the {L}andau collision operator}.
\newblock {\em Physics of Plasmas 24}, 10 (2017), 102311.

\bibitem{Larsen1985}
{\sc Larsen, E.~W., Levermore, C.~D., Pomraning, G.~C., and Sanderson, J.~G.}
\newblock Discretization methods for one-dimensional {F}okker-{P}lanck operators.
\newblock {\em J. Comput. Phys. 61}, 3 (1985), 359--390.

\bibitem{ManziniFunaroDelzanno2016}
{\sc Manzini, G., Funaro, D., and Delzanno, G.~L.}
\newblock Convergence of spectral discretizations of the vlasov-poisson system.
\newblock {\em SIAM J. Numer. Anal. 55\/} (2016), 2312--2335.

\bibitem{BorziMohammadi2015}
{\sc Mohammadi, M., and Borz\'{\i}, A.}
\newblock Analysis of the {C}hang-{C}ooper discretization scheme for a class of {F}okker-{P}lanck equations.
\newblock {\em J. Numer. Math. 23}, 3 (2015), 271--288.

\bibitem{MouhotPareschi2006}
{\sc Mouhot, C., and Pareschi, L.}
\newblock Fast algorithms for computing the {B}oltzmann collision operator.
\newblock {\em Math. Comp. 75}, 256 (2006), 1833--1852 (electronic).

\bibitem{Mouhot2013}
{\sc Mouhot, C., Pareschi, L., and Rey, T.}
\newblock {Convolutive Decomposition and Fast Summation Methods for Discrete-Velocity Approximations of the Boltzmann Equation}.
\newblock {\em ESAIM Math. Model. Numer. Anal. 47}, 5 (2013), 1515--1531.

\bibitem{PareschiRey2020}
{\sc Pareschi, L., and Rey, T.}
\newblock {On the stability of equilibrium preserving spectral methods for the homogeneous Boltzmann equation}.
\newblock {\em Applied Mathematics Letters abs/2011.05811\/} (2020).

\bibitem{PareschiRey2022}
{\sc Pareschi, L., and Rey, T.}
\newblock {Moment preserving Fourier-Galerkin spectral methods and application to the Boltzmann equation}.
\newblock {\em SIAM J. Num. Anal.}, to appear (2022).

\bibitem{Pareschi2000a}
{\sc Pareschi, L., and Russo, G.}
\newblock {Numerical Solution of the Boltzmann Equation I : Spectrally Accurate Approximation of the Collision Operator}.
\newblock {\em SIAM J. Numer. Anal. 37}, 4 (2000), 1217--1245.

\bibitem{Pareschi2000b}
{\sc Pareschi, L., and Russo, G.}
\newblock On the stability of spectral methods for the homogeneous {B}oltzmann equation.
\newblock {\em Transport Theory Statist. Phys. 29}, 3-5 (2000), 431--447.

\bibitem{Pareschi2000}
{\sc Pareschi, L., Russo, G., and Toscani, G.}
\newblock Fast spectral methods for the {F}okker-{P}lanck-{L}andau collision operator.
\newblock {\em J. Comput. Phys. 165}, 1 (2000), 216--236.

\bibitem{PareschiZanella2018}
{\sc Pareschi, L., and Zanella, M.}
\newblock Structure preserving schemes for nonlinear {F}okker-{P}lanck equations and applications.
\newblock {\em J. Sci. Comput. 74}, 3 (2018), 1575--1600.

\bibitem{ShenTangWang2011}
{\sc Shen, J., Tang, T., and Wang, L.-L.}
\newblock {\em Spectral Methods: Algorithms, Analysis and Applications}, 1st~ed.
\newblock Springer Publishing Company, Incorporated, 2011.

\bibitem{Xiu2010}
{\sc Xiu, D.}
\newblock {\em Numerical Methods for Stochastic Computations: A Spectral Method Approach}.
\newblock Princeton University Press, Princeton, 2010.

\end{thebibliography}

\end{document}